\documentclass[10pt]{amsart}
\usepackage{amsmath}
\usepackage{amsfonts}
\usepackage{amssymb}

\newcommand{\bC}{{\mathbb C}}
\newcommand{\Rst}{{\mathbb R}}

\newcommand{\Rdst}{{\Rst^d}}

\newcommand{\Rtdst}{{\Rst^{2d}}}
\newcommand{\set}[2]{\big\{ \, #1 \, \big| \, #2 \, \big\}}
\newcommand{\norm}[1]{\lVert#1\rVert}
\newcommand{\Esp}{{\boldsymbol E}}
\newcommand{\Bsp}{{\boldsymbol B}}
\newcommand{\Hsp}{{\boldsymbol H}}

\newcommand{\Zst}{{\mathbb Z}}
\newcommand{\Zdst}{{\Zst^d}}
\newcommand{\supess}{\mathop{\operatorname{supess}}}

\newcommand{\Lsp}{{\boldsymbol L}}

\newcommand{\Nst}{{\mathbb N}}
\newcommand{\Gc}{{\mathcal{G}}}
\newcommand{\pointdot}{\cdot}
\def\Xsp{{\boldsymbol X}}
\def\Esp{{\boldsymbol E}}
\def\Hsp{{\boldsymbol H}}
\newcommand{\HH}{{\mathbb{H}}}
\def\Hst{{\mathbb H}}
\def\into{\hookrightarrow}

\newcommand{\Espd}{{\boldsymbol E_d}}
\newcommand{\EspdB}{{\boldsymbol E_{d,\Bsp}}}

\newcommand{\ip}[2]{\ensuremath{\left<#1,#2\right>}}
\newcommand{\sett}[1]{\ensuremath{\left \{ #1 \right \}}}
\newcommand{\abs}[1]{\ensuremath{\left| #1 \right| }}
\newcommand{\mes}[1]{\ensuremath{\left| #1 \right| }}

\newcommand{\iv}[1]{{#1}^\vee}
\newcommand{\rel}{\rho}
\newcommand{\coef}{C}
\newcommand{\rec}{S}
\newcommand{\coefp}{C'}
\newcommand{\recp}{S'}
\newcommand{\coefv}{\coef^\Bsp}
\newcommand{\recv}{\rec^\Bsp}
\newcommand{\wweakl}{{W^{\rm weak}(L^\infty,L^1_w)}}
\newcommand{\wweak}{{W_R^{\rm weak}(L^\infty,L^1_w)}}
\newcommand{\wstrong}{{W^{\rm st}(L^\infty,L^1_w)}}
\def\Ssp{{\boldsymbol S}}
\newcommand{\SEsp}{\Ssp}
\newcommand{\stft}{\mathcal{V}}
\newcommand{\consEw}{C_{\Esp,w}}
\newcommand{\multm}{M_m}
\newcommand{\multn}{N_m}
\newcommand{\multmU}{M_{m,U}}
\newcommand{\multoneU}{M_{1,U}}

\newtheorem{lemma}{Lemma}
\newtheorem{theorem}{Theorem}

\newtheorem{prop}{Proposition}
\newtheorem{claim}{Claim}
\newtheorem{rem}{Remark}
\newtheorem{definition}{Definition}

\title{Characterization of coorbit spaces with phase-space covers}
\subjclass[2000]{42B35, 42C15, 42C40}
\keywords{Coorbit theory, localization operator,
phase-space localization, amalgam space, wavelet transform}

\author[J.L.~Romero]{Jos\'e Luis Romero}
\address{Departamento de
Matem\'atica \\ Facultad de Ciencias Exactas y Naturales\\ Universidad
de Buenos Aires\\ Ciudad Universitaria, Pabell\'on I\\ 1428 Capital
Federal\\ ARGENTINA\\ and CONICET, Argentina}
\email[Jos\'e Luis Romero]{jlromero@dm.uba.ar}
\thanks{The author acknowledges
support from the following grants: PICT 2006-00177, CONICET PIP
112-200801-00398 and UBACyT X149}

\begin{document}
\begin{abstract}
We show that coorbit spaces can be characterized in terms of
arbitrary phase-space covers, which are families of phase-space multipliers
associated with partitions of unity. This generalizes previously known
results for time-frequency analysis to include time-scale decompositions. As
a by-product, we extend the existing results for time-frequency analysis to
an irregular setting.
\end{abstract}
\maketitle
\section{Introduction}
Coorbit spaces are functional spaces defined by imposing size
conditions to a certain transform. More precisely, regarding a functional 
space $\Xsp$ as a coorbit space consist of giving a transform
$T: \Xsp \to \Esp$ that embeds $\Xsp$ into another functional space $\Esp$
that is \emph{solid}. This means that the membership in $\Esp$ is
determined by size conditions (for precise definitions see Section
\ref{sec_bf}). The space $\Esp$ consists of functions defined on a set $\Gc$
that is commonly taken to be a locally compact group.

The theory in \cite{fegr89} studies the case when $T$ arises as the
representation coefficients of a unitary action of a locally compact
group. The examples of this theory include a wide range of classical function
spaces. In the case of the affine group acting on $L^2(\Rdst)$
by translations and dilations,
$T$ is the so-called continuous wavelet transform and the corresponding class
of coorbit spaces includes the Lebesgue spaces $L^p$ ($1 < p < \infty$),
Sobolev spaces and,
more generally, the whole class of Besov and Triebel-Lizorkin spaces. In the
case of the Heisenberg group acting on $L^2(\Rdst)$ by time-frequency shifts,
the transform $T$ is known as the short-time Fourier transform (or windowed
Fourier transform) and the corresponding coorbit spaces are known as
modulation spaces \cite{fe89-1, fegr97}.

When a functional space $\Xsp$ is identified as a coorbit space, the
properties of an element $f \in \Xsp$ are reformulated in terms of decay or
integrability conditions
of the function $T(f) \in \Esp$, that is sometimes referred to as the
\emph{phase-space representation} of $f$. The elements of $\Xsp$
can be resynthesized from their phase-space representations by means
of an operator $W: \Esp \to \Xsp$ that is a left-inverse for $T$
(i.e. $f=WT(f)$).
In an attempt to finely adjust the
properties of a function $f$ that are expressed by $T(f)$ one can consider
operators of the form $M_m(f) = W(mT(f))$ that apply a mask $m$ to
the phase-space representation $T(f)$. We will call these operators
\emph{phase-space multipliers}. Of course, the rigorous
interpretation of $M_m(f)$ is problematic since, in general,
$T M_m(f) \not= m T(f)$. When $T$ is the abstract wavelet transform
(representation-coefficients function) associated with an unitary
action of a group, these operators are know as \emph{localization
operators} or \emph{wavelet multipliers}
\cite{hewo96, wong02, limowo08}. In the case of time-frequency
analysis these operators are known as time-frequency localization operators
or multipliers of the short-time Fourier transform \cite{da88, cogr03,
cogr06, bo04-2}.

In this article we characterize the norm of a coorbit space in terms
of families of phase-space multipliers associated with an arbitrary
partition of unity in $\Gc$. Specifically, suppose that $\Xsp$ is a Banach
space that is
regarded as a coorbit space by means of a transform $T:\Xsp \to \Esp$,
having a left-inverse $W:\Esp \to \Xsp$. Let
$\sett{\theta_\gamma}_\gamma$ be a partition of unity on $\Gc$ and consider
the corresponding phase-space multipliers given by $M_\gamma(f) =
W(\theta_\gamma T(f))$. The partition of unity is only assumed to
satisfy certain spatial localization conditions but it is otherwise
arbitrary.
We prove that $\norm{f}_\Xsp$ is equivalent to the norm of the sequence
$\sett{\norm{M_\gamma(f)}_\Bsp}_\gamma$ in a discrete version of the space
$\Esp$,
where the space $\Bsp$ can be chosen among a large class of function spaces.
Moreover, we prove
that the map $f \mapsto \sett{M_\gamma(f)}_\gamma$ embeds $\Xsp$ as a
complemented
subspace of a space of $\Bsp$-valued sequences, obtained as a
discretization of $\Esp$. (See Theorem \ref{th_coverings_do_cover} for a
precise statement). This quantifies the relation between an element
$f \in \Xsp$ and the phase-space localized pieces
$\sett{M_\gamma(f)}_\gamma$.

Phase-space multipliers formalize the notion of acting on a vector
by operating on its phase-space representation. The set of all phase-space
representations $T(\Xsp)$ is usually thought of as the class of all functions
on phase-space, and phase-space itself is implicitly understood as the
underlying set of ``degrees of freedom'' for that class of functions.
Operations on phase-space are formally described as operations on the class
of functions $T(\Xsp)$. Thus, a family of phase-space multipliers
$\sett{M_\gamma}_\gamma$ associated with a partition of unity
$\sett{\theta_\gamma}_\gamma$ on $\Gc$ is a natural formalization of the
notion of a cover of phase-space. From this perspective, the estimates we
prove, that establish a quantitative equivalence between a vector $f$
and the sequence $\sett{M_\gamma(f)}_\gamma$, can be interpreted as saying
that the family of operators $\sett{M_\gamma}_\gamma$ indeed covers
phase-space.

For the case of time-frequency analysis, D\"orfler, Feichtinger and
Gr\"ochenig \cite{dofegr06, dogr09} have recently obtained a
characterization of modulation spaces through families of time-frequency
localization operators, using techniques from rotation algebras
(non-commutative tori) developed in \cite{grle04} and \cite{gr07-2} and
spectral theory for Hilbert spaces.\footnote{For more about the relation
between time-frequency analysis and non-commutative tori see \cite{lu09}.}
In this article we use a different approach to obtain
consequences for settings where the techniques in \cite{dogr09} are not
applicable, such as time-scale decompositions and Besov spaces.
As a by-product we derive a stronger version of the main result in
\cite{dogr09} where the admissible partitions of unity are restricted
to be lattice shifts of a non-negative function and the space $\Bsp$
is $L^2$. (For precise statements see Section \ref{sec_app_gab}).

We now comment on the organization of the article.
We consider an abstract setting in which there is a solid space
$\Esp$ of functions over a group $\Gc$ and a certain complemented subspace
$\SEsp$ (this is similar to the setting studied in \cite{nasu10}).
Phase-space multipliers are defined as operators of the form
$\SEsp \ni f \mapsto P(mf) \in \SEsp$ where
$P:\Esp \to \SEsp$ is a projection and $m \in L^\infty(\Gc)$.
The main result we prove is the characterization of the norm
of $\SEsp$ in terms of the family of multipliers associated with
an arbitrary partition of unity in $\Gc$ (see Theorem
\ref{th_coverings_do_cover}).
The technique of the proof is a vector-valued variant of the
proof of the existence of atomic decompositions for coorbit spaces in
\cite{fegr89}. In Section \ref{sec_applications} we apply the abstract
results to coorbit spaces, by taking $\SEsp$ to be the range of an adequate
transform. The model for phase-space that we consider includes the classical
coorbit theory associated with a group representation but also the case of
coorbit spaces produced from localized frames (see \cite{fogr05}). This
yields further applications to time-frequency analysis, giving a
characterization of modulation spaces in terms of certain discrete
time-frequency localization operators known as Gabor multipliers (see
\cite{feno03}).

Amalgam spaces (see Section \ref{sec_amalgams}) are one of the main
technical tools of this article. We use and slightly extend (see Section
\ref{sec_weak_am}) a number of convolution and sampling
relations from \cite{fegr89} that are particularly important to
discretization of convolution operators.

Section \ref{sec_more_general} establishes a variant of the main result
where, under stronger assumptions on the group $\Gc$, the class
of admissible partitions of unity is enlarged. This partial extension
of the main result is important in a number of examples
and, in particular, allows us to recover and extend the main result from
\cite{dogr09}.
Instead of the tools from rotation algebras used there, we resort to related
results for matrix algebras.

\section{Preliminaries}
\subsection{Notation}
Throughout the article $\Gc$ will be a locally compact, $\sigma$-compact,
topological group with identity element $e$ and modular function $\Delta$.
The left Haar measure of a set $X \subseteq \Gc$ will be denoted by $\mes{X}$
whereas its cardinality will be denoted by $\# X$. Integration will be always
considered with respect to the left Haar measure. The product of two subsets of
$\Gc$, $A,B$, will be denoted $A \cdot B$ or simply $AB$.

For $x \in \Gc$, we denote by $L_x$ and $R_x$ the operators of
left and right translation, defined by
$L_x f(y) = f(x^{-1}y)$ and $R_x f(y) = f(yx)$.
We also consider the involution $\iv{f}(x) = f(x^{-1})$.

Given two non-negative functions $f,g$ we write $f \lesssim g$ if there
exists a constant $C \geq 0$ such that $f \leq C g$. We say that $f \approx
g$ if both $f \lesssim g$ and $g \lesssim f$.
The characteristic function of the set $A$ will be denoted by $\chi_A$.
The symbol $\ip{\cdot}{\cdot}$ will stand for the $L^2$ inner product,
$\ip{f}{g}:=\int_\Gc f(x) \overline{g}(x) dx$, whenever defined.

A set $\Lambda \subseteq \Gc$ is called \emph{relatively separated} if 
for some (or any) $V \subseteq \Gc$, relatively compact neighborhood of $e$,
the quantity - called the \emph{spreadness of $\Lambda$} -
\begin{align}
\label{eq_spread}
\rel(\Lambda) = \rel_V(\Lambda)
:= \sup_{x \in \Gc} \# (\Lambda \cap x V)
\end{align}
is finite, i.e. if the amount of elements of $\Lambda$ that lie in any left
translate of $V$ is uniformly bounded. Equivalently, $\Lambda$ is relatively
separated if for any compact set $K \subseteq \Gc$,
\begin{align*}
\sup_{\lambda \in \Lambda}
\#\set{\lambda' \in \Lambda} {\lambda K \cap \lambda' K \neq \emptyset}
<+\infty.
\end{align*}
A set $\Lambda \subseteq \Gc$ is called $V$-\emph{dense} (for $V$, a
relatively compact neighborhood of $e$) if
$\Gc = \bigcup_{\lambda \in \Lambda} \lambda V$.
$\Lambda$ is called \emph{well-spread} if it is both relatively separated and
$V$-dense for some $V$.

We now fix $V$, a symmetric (i.e. $V=V^{-1}$) relatively compact neighborhood
of the identity in $\Gc$. Some definitions below depend on the choice of
$V$, but different choices of $V$ will yield equivalent objects.

We will sometimes assume that $\Gc$ is an IN group, i.e., that it has a
relatively compact neighborhood of the identity that is invariant under
inner automorphisms. By convention, whenever we assume that $\Gc$ is an IN
group we will further assume that the distinguished neighborhood $V$ is
invariant (i.e. $xVx^{-1}=V$, for all $x \in \Gc$).

\subsection{Function spaces}
\label{sec_bf}
A \emph{BF space} is a Banach space $\Esp$ consisting of functions on
$\Gc$ that is continuously embedded into $L_{\mathrm{loc}}^1(\Gc)$,
the space of locally integrable functions.

A BF space $\Esp$ is called \emph{solid} if for every $f \in \Esp$
and every measurable function $g: \Gc \to \bC$ such that
$\abs{g(x)} \leq \abs{f(x)}$ a.e., it is true that $g \in \Esp$
and $\norm{g}_\Esp \leq \norm{f}_\Esp$.

An \emph{admissible weight} is a locally bounded function
$w: \Gc \to (0,+\infty)$ that satisfies the following conditions,
\begin{align}
\label{weight_w_delta}
&w(x) = \Delta(x^{-1})w(x^{-1}),
\\
\label{weight_w_submult}
&w(xy) \leq w(x)w(y)
\mbox{ (submultiplicativity)}.
\end{align}
If $\Esp$ is a solid BF space and $w$ is an admissible weight,
we let $\Esp_w$ be the set of all functions $f \in L_{\mathrm{loc}}^1(\Gc)$
such that $fw \in \Esp$ and endow it with the norm
$\norm{f}_{\Esp_w} := \norm{fw}_\Esp$.
If $w$ is an admissible weight, then
$L^1_w$ is a convolution algebra,
$\norm{f}_{L^1_w}=\norm{\iv{f}}_{L^1_w}$, and
$\norm{L_x}_{L^1_w \to L^1_w} \leq w(x)$.

We say that a solid BF space $\Esp$ is \emph{translation invariant} if
it satisfies the following.
\begin{itemize}
\item[(i)] $\Esp$ is closed under left and right translations
(i.e. $L_x \Esp \subseteq \Esp$ and $R_x \Esp \subseteq \Esp$, for all $x \in
\Gc$).
\item[(ii)] The relations,
\begin{align}
\label{conv_mod_assumption}
L^1_u(\Gc)*\Esp \subseteq \Esp
\mbox{  and  }
\Esp*L^1_v \subseteq \Esp,
\end{align}
hold, with the corresponding norm estimates,
where $u(x) := \norm{L_x}_{\Esp \to \Esp}$
and $v(x) := \Delta(x^{-1})\norm{R_{x^{-1}}}_{\Esp \to \Esp}$.
\end{itemize}
\begin{rem}
Observe that, for a BF space, if the translations leave $\Esp$ invariant,
then they are bounded by the closed graph theorem.
\end{rem}
\begin{rem}
In the definition of translation invariant space, the
technical assumption (ii) follows from (i)
if the set of continuous functions with compact
support is dense on $\Esp$, or more generally if
the maps $x \mapsto L_x$ and $x \mapsto R_x$
are strongly continuous.
\end{rem}
We say that $\Esp$ is isometrically left (right) translation
invariant if it is translation invariant and, in addition,
left (right) translations are isometries on $\Esp$.
The weighted Lebesgue spaces $L^p_m(\Rdst)$ with
$m(x) := (1 + \abs{x})^{\alpha}$, $\alpha \in \Rst$, are examples
of translation invariant solid BF spaces on $\Rdst$. These are isometrically
translation invariant if $m \equiv 1$.

Given a BF space $\Esp$, a set of functions
$\set{f_\lambda}{\lambda \in \Lambda} \subseteq L_{\mathrm{loc}}^1(\Gc)$
- indexed by a relatively separated set $\Lambda$ - is called a set of
$\Esp$-\emph{molecules}
if there exists a function $g \in \Esp$ - called envelope - such that
\begin{align*}
 \abs{f_\lambda(x)} \leq L_\lambda g (x),
\qquad (x \in \Gc, \lambda \in \Lambda).
\end{align*}

Given a solid, translation invariant, BF space $\Esp$, we
say that a weight $w: \Gc \to (0,+\infty)$ is
\emph{admissible for $\Esp$} if $w$ is admissible and,
in addition, it satisfies,
\begin{align}
\label{weight_w_admissible}
w(x) \geq \consEw \max \sett{u(x), \; u(x^{-1}), \; v(x),
\; \Delta(x^{-1})v(x^{-1})},
\end{align}
where $u(x) := \norm{L_x}_{\Esp \to \Esp}$,
$v(x) := \Delta(x^{-1})\norm{R_{x^{-1}}}_{\Esp \to \Esp}$
and $\consEw > 0$ is a constant.
Under these conditions, $w(x) \gtrsim 1$,
$L^1_w * \Esp \subseteq \Esp$
and $\Esp*L^1_w \subseteq \Esp$, with the corresponding norm estimates.
Moreover, the constants in those estimates depend only on $\consEw$.

If $\Esp$ is a solid BF space, we construct discrete versions
of it as follows. Given a well-spread set $\Lambda \subseteq \Gc$
we define the space,
\begin{align*}
\Espd = \Espd(\Lambda) := \set{c \in {\bC}^\Lambda}
{\sum_\lambda \abs{c_\lambda} \chi_{\lambda V} \in \Esp},
\end{align*}
and endow it with the norm,
\begin{align*}
 \norm{\left(c_\lambda \right)_{\lambda \in \lambda}}_{\Espd}
 := \norm{\sum_\lambda \abs{c_\lambda} \chi_{\lambda V}}_\Esp.
\end{align*}
The definition, of course, depends on $\Lambda$ and $V$,
but a large class of neighborhoods $V$ and sets $\Lambda$
produce equivalent spaces (see \cite[Lemma 3.5]{fegr89} for a precise
statement). In the sequel, we will mainly use the space $\Espd$ keeping
$V$ fixed and making an explicit choice for $\Lambda$.
When $\Esp=L^p_w$, for an admissible weight $w$, the corresponding discrete
space $\Espd(\Lambda)$ is $\ell^p_w(\Lambda)$, where the weight $w$ is
restricted to the set $\Lambda$. This is so because the admissibility of $w$
implies that for $x \in \lambda V$, $w(x) \approx w(\lambda)$.

We will also need a vector-valued version of $\Espd$. Given a second Banach
space $\Bsp$ we let,
\begin{align*}
\EspdB = \EspdB(\Lambda) 
&:= \set{c \in {\Bsp}^\Lambda}
{ (\norm{c_\lambda}_\Bsp)_{\lambda \in \Lambda} \in \Espd(\Lambda)}
\\
&= \set{c \in {\Bsp}^\Lambda}
{\sum_\lambda \norm{c_\lambda}_\Bsp \chi_{\lambda V} \in \Esp},
\end{align*}
and endow it with a norm in a similar fashion.

\subsection{Wiener amalgam spaces}
\label{sec_amalgams}\cite{fe83, ho75, foust85}.
Given two solid, translation invariant BF spaces $\Bsp$ and $\Esp$,
the left \emph{amalgam space} (or space of Wiener-type) with local component
$\Bsp$ and global component $\Esp$ is defined by,
\begin{align*}
W(\Bsp, \Esp) := \set{f \in \Bsp_{\textit{loc}}}{K_\Bsp(f) \in \Esp},
\end{align*}
where the control function $K_\Bsp(f)$ is given by
$K_\Bsp(f)(x) := \norm{f (L_x \chi_V)}_\Bsp=\norm{f \chi_{xV}}_\Bsp$.
We endow $W(\Bsp, \Esp)$ with the norm
$\norm{f}_{W(\Bsp, \Esp)} := \norm{K_\Bsp(f)}_\Esp$.
The right amalgam space $W_R(\Bsp, \Esp)$ is defined similarly,
this time using the control function
$K_\Bsp(f,R)(x):= \norm{f (R_x \chi_V)}_\Bsp=\norm{f \chi_{Vx^{-1}}}_\Bsp$.
This definition depends on the choice of the neighborhood $V$, but a
different choice produces the same space with an equivalent norm.

When $\Bsp$ and $\Esp$ are weighted $L^p$ spaces, the corresponding
amalgam space coincides with the classical $L^p-\ell^q$ amalgam spaces
\cite{ho75, foust85}. We are requiring that the space $\Bsp$ be solid,
but much of the theory only requires that $\Bsp$ have a sufficiently rich
algebra of pointwise multipliers (see \cite{fe83}).

Amalgam spaces with $L^1$ and $L^\infty$ as local components
will be a key technical tool in this article. The spaces
$W(L^\infty,\Esp)$ and $W_R(L^\infty,\Esp)$ can
be easily described in terms of certain maximum functions.
For a locally bounded function $f: \Gc \to \bC$ we define the
left and right \emph{local maximum} functions by,
\begin{align*}
f^\# (x) &:= \supess_{y \in V} \abs{f(xy)},
\\
f_\# (x) &:= \supess_{y \in V} \abs{f(yx)}.
\end{align*}
Since $V$ is symmetric,
these functions are related by $\iv{(f_\#)}=(\iv{f})^\#$.
Using these definitions we have,
\begin{align}
\label{eq_left_am}
\norm{f}_{W(L^\infty, \Esp)} &= \norm{f^\#}_{\Esp},
\\
\label{eq_right_am}
\norm{f}_{W_R(L^\infty, \Esp)} &= \norm{\iv{(f_\#)}}_{\Esp}.
\end{align}
In particular,
$\norm{f}_{W_R(L^\infty, \Esp)} =
\norm{\iv{f}}_{W(L^\infty, \Esp)}$. Note also that,
by the solidity of $\Esp$, both $W(L^\infty, \Esp)$ and $W_R(L^\infty, \Esp)$
are continuously embedded into $\Esp$. We will denote by $W(C_0,\Esp)$
the subspace of $W(L^\infty, \Esp)$ formed by the continuous functions.
$W_R(C_0,\Esp)$ is defined similarly.

When $\Esp=L^1_w$ for an admissible weight $w$
we can drop the involution in Equation \eqref{eq_right_am} yielding, 
\begin{align}
\label{deff_WR}
\norm{f}_{W_R(L^\infty, L^1_w)}
:= \norm{f_\#}_{L^1_w}.
\end{align}
In addition, since $L_x (f^\#)= (L_x f)^\#$, the space
$W(L^\infty,L^1_w)$ is invariant under left translations
and the norm of
the left translations is dominated by $w$. A similar statement
holds for $W_R(L^\infty,L^1_w)$ and right translations.

We finally note that if $\Gc$ is an IN group, the left and
right local maximum functions coincide and therefore
$W(L^\infty, \Esp)=W_R(L^\infty, \Esp)$.

We now state some facts about amalgam spaces
that are relevant to atomic decompositions.
They have been mainly collected from \cite{fegr89}.
In the cases when we were unable to find an exact reference we sketch a
proof. Most of the results below concern a translation invariant BF space
$\Esp$ and an admissible weight $w$. We point out that the corresponding
estimates depend only on $w$ and certain qualities of $\Esp$, 
namely the value of the constant $\consEw$ in Equation
\eqref{weight_w_admissible}.
\begin{lemma}
\label{lemma_amfacts1}
Let $\Esp$ be a solid, translation invariant BF space
and let $w$ be an admissible weight for it.
The following embeddings hold, together with the corresponding
norm estimates.\footnote{Here, the symbol $\cdot$ denotes the pointwise
product.}
\begin{itemize}
\item[(a)] $\Esp *W(L^\infty,L^1_w)
\hookrightarrow W(L^\infty,\Esp)$
and $\Esp *W(C_0,L^1_w)
\hookrightarrow W(C_0,\Esp)$.
\item[(b)] $\Esp \hookrightarrow W(L^1,L^\infty_{1/w})$.
In addition, if $\Esp$ is isometrically left-translation invariant,
then $\Esp \hookrightarrow W(L^1,L^\infty)$.
\item[(c)] $W(L^1,L^\infty) \pointdot W(L^\infty,L^1) \hookrightarrow L^1$
and $W(L^1,L^\infty_{1/w}) \pointdot W(L^\infty,L^1_w) \hookrightarrow L^1$.
\item[(d)] $W(L^1,L^\infty)*W_R(L^\infty,L^1)
\hookrightarrow L^\infty$ and
$W(L^1,L^\infty_{1/w})*W_R(L^\infty,L^1_w)
\hookrightarrow L^\infty_{1/w}$.
\end{itemize}
\end{lemma}
\begin{proof}
Part (a) is proved in \cite[Theorem 7.1]{fegr89-1}.
By \cite[Lemma 3.9]{fegr89},
$\Esp \hookrightarrow W(L^1,L^\infty_{(\iv{u})^{-1}})$,
where $u(x) := \norm{L_x}_{\Esp \to \Esp}$. The admissibility
of $w$ implies that $\iv{u} \lesssim w$, so part (b) follows.

To prove (c) first observe that for any $f \in L^1(\Gc)$, since $V=V^{-1}$,
\begin{align*}
\int_\Gc \int_\Gc \abs{f(x)} (L_y \chi_V(x)) dx dy
&= \int_\Gc \abs{f(x)} \int_\Gc \chi_V(y^{-1}x) dy dx
= \mes{V} \int_\Gc \abs{f(x)} dx.
\end{align*}
Using this observation, for $f \in W(L^1,L^\infty_{1/w})$
and $g \in W(L^\infty,L^1_w)$,
\begin{align*}
&\int_\Gc \abs{f(x)} \abs{g(x)} dx
\approx
\int_\Gc \int_{yV} \abs{f(x)} \abs{g(x)} dx dy
\\
&\quad \lesssim
\int_\Gc \norm{f}_{L^1(yV)} \norm{g}_{L^\infty(yV)} dy
\leq \norm{f}_{W(L^1,L^\infty_{1/w})} \norm{g}_{W(L^\infty,L^1_w)}.
\end{align*}
The unweighted case follows similarly.
To prove (d) let $f \in W(L^1,L^\infty_{1/w})$ and
$g \in W_R(L^\infty,L^1_w)$.
For $x \in \Gc$ we can use (c) to get,
\begin{align*}
\abs{f*g(x)} &\leq \int_\Gc \abs{f(y)}\abs{L_x \iv{g}(y)} dy
\leq
\norm{f}_{W(L^1,L^\infty_{1/w})} \norm{L_x\iv{g}}_{W(L^\infty,L^1_w)}
\\
&\leq \norm{f}_{W(L^1,L^\infty_{1/w})} 
\norm{\iv{g}}_{W(L^\infty,L^1_w)} w(x).
\end{align*}
Since $\norm{\iv{g}}_{W(L^\infty,L^1_w)}=\norm{g}_{W_R(L^\infty,L^1_w)}$
the weighted inequality in (d) is proved. The unweighted one
follows similarly, this time using the unweighted bound in (c).
\end{proof}

\begin{lemma}
\label{lemma_amfacts2}
Let $\Esp$ be a solid, translation invariant BF space, 
let $w$ be an admissible weight for it and let
$\Lambda \subseteq \Gc$ be a relatively separated set.
Then the following holds.
\begin{itemize}
\item[(a)]  For every $f \in W(C_0, \Esp)$,
the sequence $f(\Lambda) = (f(\lambda))_{\lambda \in \Lambda}$ belongs
to $\Espd(\Lambda)$ and,
\[\norm{f(\Lambda)}_{\Espd} \lesssim \norm{f}_{W(C_0, \Esp)}.\]
\item[(b)] For every $f \in \Esp$ and $g \in W_R(C_0,L^1_w)$,
the sequence $(\ip{f}{L_\lambda g})_{\lambda \in \Lambda}$
belongs to $\Espd(\Lambda)$ and,
\[\norm{(\ip{f}{L_\lambda g})_\lambda}_{\Espd}
\lesssim \norm{f}_\Esp \norm{g}_{W_R(L^\infty,L^1_w)}.\]
\item[(c)]
If $(c_\lambda)_\lambda \in \Espd(\Lambda)$
and $f \in W_R(L^\infty,L^1_w)$,
then $\sum_\lambda c_\lambda L_\lambda f \in \Esp$ and
\[\norm{\sum_\lambda c_\lambda L_\lambda f}_\Esp
\lesssim \norm{(c_\lambda)_\lambda}_{\Espd} \norm{f}_{W_R(L^\infty,
L^1_w)}.\]
The series converges absolutely at every point and,
if the set of bounded compactly supported functions is dense in $\Esp$,
it also converges unconditionally in the norm of $\Esp$.
\item[(d)] $\Espd(\Lambda) \hookrightarrow \ell^\infty_{1/w}(\Lambda)$.
\end{itemize}
All the implicit constants depend on $\rel(\Lambda)$ - the spreadness of
$\Lambda$ (cf. Equation \eqref{eq_spread}).
\end{lemma}
\begin{proof}
Part (a) follows easily from the definitions (see for
example \cite[Lemma  3.8]{fegr89}).
For (b) observe that
$\ip{f}{L_\lambda g} = (f * \iv{g})(\lambda)$. Hence,
Lemma \ref{lemma_amfacts1} and part (a) imply that
\begin {align*}
\norm{(\ip{f}{L_\lambda g})_\lambda}_{E^d}
\lesssim \norm{f * \iv{g}}_{W(C_0,\Esp)}
\lesssim \norm{f}_\Esp \norm{\iv{g}}_{W(L^\infty,L^1_w)}
= \norm{f}_\Esp \norm{g}_{W_R(L^\infty,L^1_w)}.
\end {align*}

Part (c) is Proposition 5.2 of \cite{fegr89}. Lemma 3.5
in \cite{fegr89} gives the embedding
$\Espd(\Lambda) \into \ell^\infty_{1/u}(\Lambda)$,
where $u(x) := \norm{L_x}_{\Esp \to \Esp}$. Since
$u \lesssim w$, part (d) follows.
\end{proof}
Finally we state the following lemma that will be used
to justify treating convolutions pointwise.
\begin{lemma}
\label{lemma_amfacts25}
Let $\Esp$ be a solid, translation invariant BF space
and let $w$ be an admissible weight for it.
The following embeddings hold, together with the corresponding
norm estimates.
\begin{itemize}
\item[(a)] $W(L^\infty, \Esp) \into L^\infty_{1/w}$.
\item[(b)] $W(L^\infty, \Esp) * L^1_w \into C_{1/w}$,
where $C_{1/w}$ denotes the subspace of $L^\infty_{1/w}$
formed by the continuous functions.
\end{itemize}
\end{lemma}
\begin{proof}
By Lemma \ref{lemma_amfacts2}, $\Espd \into \ell^\infty_{1/w}$.
This implies that,
\begin{align*}
W(L^\infty, \Esp) \into
W(L^\infty, L^\infty_{1/w})=L^\infty_{1/w},
\end{align*}
(see for example \cite[Proposition 3.7]{fegr89}). This proves part (a).
The embedding,
\begin{align*}
W(L^\infty, \Esp) \pointdot L^1_w
\into L^\infty_{1/w} \pointdot L^1_w \into L^1, 
\end{align*}
implies that,
$W(L^\infty, \Esp)*L^1_w = W(L^\infty, \Esp)*{\iv{L^1_w}}
\into L^\infty_{1/w}$. Now part (b) follows from the fact
that the class of continuous,
compactly supported functions is dense in $L^1_w$.
\end{proof}

\subsection{Weak and strong amalgam norms}
\label{sec_weak_am}
We now introduce some variations of the amalgam
spaces $W(L^\infty,L^1_w)$, $W_R(L^\infty,L^1_w)$.
We do so in order to deal with some technicalities
involving right convolution actions on the spaces
$W(L^\infty, \Esp)$. For an IN group,
the spaces $W(L^\infty, \Esp)$ are right $L^1_w$ modules,
but for a general group $\Gc$, they are only right
$W(L^\infty,L^1_w)$ modules. We will now introduce
a space between $L^1_w$ and $W(L^\infty,L^1_w)$
that acts on the spaces $W(L^\infty, \Esp)$ from the right
and collapses to $L^1_w$ in the case that $\Gc$ is an IN group.

Similarly, we will define a certain subspace of
$W(L^\infty,L^1_w) \cap W_R(L^\infty,L^1_w)$ that reduces to
$W(L^\infty,L^1_w)$ when $\Gc$ is an IN group. The introduction
of this second space is not essential but a matter of convenience.
Its use is not required by any of the applications in Section
\ref{sec_applications}.

For an admissible weight $w$,
let the left and right \emph{weak amalgam spaces} be defined by
\begin{align*}
\wweakl &:= \set{f \in L_{\mathrm{loc}}^1}{\chi_V*\abs{f} \in
W(L^\infty,L^1_w)},
\\
\wweak &:= \set{f \in L_{\mathrm{loc}}^1}{\abs{f}*\chi_V \in
W_R(L^\infty,L^1_w)},
\end{align*}
and endow them with the norms,
\begin{align*}
\norm{f}_\wweakl &:= \norm{\chi_V*\abs{f}}_{W(L^\infty,L^1_w)}
= \norm{(\chi_V*\abs{f})^\#}_{L^1_w},
\\
\norm{f}_\wweak &:= \norm{\abs{f}*\chi_V}_{W_R(L^\infty,L^1_w)}
= \norm{(\abs{f}*\chi_V)_\#}_{L^1_w}.
\end{align*}
These spaces are related by
$\norm{f}_\wweakl = \norm{\iv{f}}_\wweak$.

Consider also the \emph{strong amalgam space} defined as,
\begin{align*}
\wstrong := W_R(L^\infty, W(L^\infty,L^1_w)).
\end{align*}
Hence, the norm of a function $f \in \wstrong$ is given by,
\begin{align*}
\norm{f}_\wstrong = \norm{(f_\#)^\#}_{L^1_w}.
\end{align*}
We now observe how these new spaces are related to the classical ones.
\begin{prop}
\label{prop_weak_st_embeddings}
Let $w$ be an admissible weight. Then the following holds.
\begin{itemize}
\item[(a)]
\[W(L^\infty,L^1_w) \hookrightarrow \wweakl \hookrightarrow L^1_w,\]
and
\[W_R(L^\infty,L^1_w) \hookrightarrow \wweak \hookrightarrow L^1_w.\]
\item[(b)] If $\Gc$ is an IN group then, 
\begin{align*}
\wweak =\wweakl= L^1_w.
\end{align*}
\item[(c)] $\wstrong \hookrightarrow W(L^\infty,L^1_w)
\cap W_R(L^\infty,L^1_w)$.
\item[(d)] If $\Gc$ is an IN group then, 
\[W(L^\infty,L^1_w)=W_R(L^\infty,L^1_w)=\wstrong.\]
\end{itemize}
\end{prop}
\begin{proof}
For (a) and (b) we only prove the statements concerning
the ``right'' spaces; the corresponding statements
for ``left'' spaces follow by using the involution $\iv{}$.

Let $f \in W_R(L^\infty,L^1_w)$.
Since $(\abs{f}*\chi_V)_\# \leq (f_\# * \chi_V)$, we have that,
\begin{align*}
\norm{f}_\wweak &= \norm{(\abs{f}*\chi_V)_\#}_{L^1_w}
\leq \norm{f_\# * \chi_V}_{L^1_w}
\\
&\leq \norm{f_\#}_{L^1_w} \norm{\chi_V}_{L^1_w}
\lesssim \norm{f}_{W_R(L^\infty,L^1_w)}.
\end{align*}
This proves the first embedding of (a). For the second one,
let $f \in \wweak$ and estimate,
\begin{align*}
\int_\Gc \abs{f(x)} w(x) dx
&\lesssim
\int_\Gc \abs{f(x)} w(x) \int_\Gc \chi_V(x^{-1}y) dy dx
\\
&\leq \int_\Gc \int_\Gc \abs{f(x)} w(y^{-1}x) \chi_V(x^{-1}y) dx w(y) dy.
\end{align*}
Since $w$ is locally bounded, in the last integral $w(y^{-1}x) \lesssim 1$
and we conclude that
$\norm{f}_{L^1_w} \lesssim \norm{\abs{f}*\chi_V}_{L^1_w}$.
Now the conclusion follows from the fact that
$\abs{f}*\chi_V \leq (\abs{f}*\chi_V)_\#$.

Part (b) follows from the convolution relation,
\begin{align*}
W(L^\infty, L^1_w) * L^1_w \hookrightarrow W(L^\infty, L^1_w),
\end{align*}
which holds when $\Gc$ is an IN group.
(This follows easily from the fact that, for an IN group,
$f^\#=f_\#$,
see for example \cite[Theorem 3]{fe83})\footnote{Theorem 3 in
\cite{fe83} implies that
$W(L^\infty, L^1_w) * W(L^1,L^1_w) \hookrightarrow W(L^\infty, L^1_w)$.
It is straightforward to see that $W(L^1,L^1_w)=L^1_w$.}.

Part (c) follows from the observation that 
$f_\# \leq (f_\#)^\#$ and $f^\# \leq (f^\#)_\#$.
Finally if $\Gc$ is an IN group, for $x \in \Gc$,
$VxV=VVx$, and therefore,
\begin{align}
(f_\#)^\#(x) = \sup_{v \in V} f_\#(xv)
= \sup_{v \in V} \sup_{w \in V} \abs{f(wxv)}
= \sup_{y \in VV} \abs{f(yx)} = (f^\#)_\#(x).
\end{align}
Hence, the conclusion follows from the fact that a different choice
for the neighborhood $V$ induces an equivalent norm in $W(L^\infty, L^1_w)$.
\end{proof}
For the weak norm, we now derive the following convolution relation
(cf. Lemma \ref{lemma_amfacts1}). Again, we point out that the estimates
depend only on the weight $w$ and the constant $\consEw$ in Equation
\eqref{weight_w_admissible}.
\begin{prop}
\label{prop_weak_conv}
Let $\Esp$ be a solid, translation invariant BF space
and let $w$ be an admissible weight for it. Then,
\begin{align*}
W(L^\infty,\Esp)*\wweakl \hookrightarrow W(C_0,\Esp),
\end{align*}
together with the corresponding norm estimate.
\end{prop}
\begin{proof}
Let $f \in W(L^\infty,\Esp)$ and $g \in \wweakl$.
For almost every $y \in \Gc$ and $t \in V$, 
$\abs{f(y)} \leq f^\#(yt)$.
Hence, for $x \in \Gc$,
\begin{align*}
\abs{f}*\abs{g}(x)
&\leq
\int_\Gc \int_\Gc f^\#(yt) \chi_V(t^{-1}) dt \abs{g(y^{-1}x)} dy
\\
&= \int_\Gc f^\#(t) \int_\Gc \chi_V(t^{-1}y) \abs{g(y^{-1}x)} dy dt
\\
&= \int_\Gc f^\#(t) (\chi_V*\abs{g})(t^{-1}x) dt
=  f^\# * (\chi_V*\abs{g}) (x).
\end{align*}
Therefore Lemma \ref{lemma_amfacts1} implies that,
\begin{align*}
&\norm{f*g}_{W(L^\infty, \Esp)}
\leq
\norm{f^\# * (\chi_V*\abs{g})}_{W(L^\infty, \Esp)}
\\
&\qquad \lesssim
\norm{f^\#}_\Esp \norm{\chi_V*\abs{g}}_{W(L^\infty, L^1_w)}
=
\norm{f}_{W(L^\infty,\Esp)}
\norm{g}_\wweakl.
\end{align*}
It only remains to note that $f*g$ is a continuous function.
This follows from the embedding $\wweakl \into L^1_w$ in Proposition
\ref{prop_weak_st_embeddings} and Lemma \ref{lemma_amfacts25}.
\end{proof}
Using Proposition \ref{prop_weak_conv}, we can derive a 
variant of Lemma \ref{lemma_amfacts2} (b) that only
requires $g$ to be in $\wweak$.
\begin{lemma}
\label{lemma_amfacts3}
Let $\Esp$ be a solid, translation invariant BF space
and let $w$ be an admissible weight for it.
Let $\Lambda \subseteq \Gc$ be a relatively separated set. Then,
for $f \in W(L^\infty,E)$ and $g \in \wweak$, the sequence
$(\ip{f}{L_\lambda g})_{\lambda \in \Lambda}$ belongs to $\Espd(\Lambda)$
and satisfies
\begin{align*}
\norm{(\ip{f}{L_\lambda g})_\lambda}_{E^d}
\lesssim \norm{f}_{W(L^\infty,E)} \norm{g}_\wweak,
\end{align*}
where the implicit constant depends on $\rel(\Lambda)$
(cf. Equation \eqref{eq_spread}).
\end{lemma}
\begin{proof}
As in the proof on Lemma \ref{lemma_amfacts2} (b),
$\norm{(\ip{f}{L_\lambda g})_\lambda}_{E^d}
\lesssim \norm{f*\iv{g}}_{W(C_0,\Esp)}$.
Now we can invoke Proposition \ref{prop_weak_conv}
and the fact that the involution $\iv{}$ maps
$\wweak$ into $\wweakl$ to obtain the desired conclusion.
\end{proof}

\section{The model for phase-space}
\label{sec_model}
We now introduce a general setting where there is a solid
BF space $\Esp$ (called the environment) and a certain distinguished subspace
$\SEsp$ that is the range of an idempotent integral operator
$P$.\footnote{This
is similar to the setting studied in \cite{nasu10}.}
This is the natural setting for the results of this article
and seems to be the easiest scenario to check in a number of concrete
examples (see Section \ref{sec_applications}).
In Section \ref{sec_setting_atdesc} we will consider a more particular
setting where the subspace $\SEsp$ has a distinguished atomic decomposition.
This will allow us to make fine adjustments to
the general results, as required by certain applications
(see Section \ref{sec_more_general}).

We list a number of ingredients in the form of two assumptions:
(A1) and (A2).
\begin{itemize}
\item[(A1)]
\begin{itemize}
\item  $\Esp$ is a solid, translation invariant BF space,
called \emph{the environment}.
\item $w$ is an admissible weight for $\Esp$.
\item $\SEsp$ is a closed complemented subspace of $\Esp$,
called \emph{the atomic subspace}.
\end{itemize}
\end{itemize}
The second assumption is that
the retraction $\Esp \to \SEsp$ is given by an operator
that is dominated by right convolution with a kernel in
$W(L^\infty,L^1_w) \cap W_R(L^\infty,L^1_w)$.
\begin{itemize}
\item[(A2)] We have an operator $P$ and a function $H$ satisfying
the following.
\begin{itemize}
\item $P: W(L^1,L^\infty_{1/w}) \to L^\infty_{1/w}$ is
a (bounded) linear operator,
\item $P(\Esp) = \SEsp$,
\item $P(f) = f, \mbox{for all } f \in \SEsp$,
\item $H \in W(L^\infty,L^1_w) \cap W_R(L^\infty,L^1_w)$,
\item For $f \in W(L^1,L^\infty_{1/w})$,
\begin{align}
\label{eq_P_dominated}
\abs{P(f)(x)} \leq \int_{\Gc} \abs{f(y)} H(y^{-1}x) dy,
\qquad (x \in \Gc).
\end{align}
\end{itemize}
\end{itemize}
We now observe some consequences of these assumptions.
\begin{prop}
\label{prop_P_into_am}
Under Assumptions (A1) and (A2) the following holds.
\begin{itemize} 
\item[(a)] $P$ boundedly maps $\Esp$ into $W(L^\infty,\Esp)$.
\item[(b)] $\SEsp \hookrightarrow W(L^\infty, \Esp)$.
\item[(c)] If $f \in W(L^1,L^\infty_{1/w})$, then
$\norm{P(f)}_{L^\infty_{1/w}} \lesssim
\norm{f}_{W(L^1,L^\infty_{1/w})} \norm{H}_{W_R(L^\infty,L^1_w)}$.
\item[(d)] If $f \in W(L^1,L^\infty)$, then
$\norm{P(f)}_{L^\infty} \lesssim
\norm{f}_{W(L^1,L^\infty)} \norm{H}_{W_R(L^\infty,L^1_w)}$.
\end{itemize}
\begin{rem}
Since $w \gtrsim 1$, $L^\infty \into L^\infty_{1/w}$.
\end{rem}
\end{prop}
\begin{proof}
Part (a), (c) and (d) follow from Equation \eqref{eq_P_dominated},
Lemma \ref{lemma_amfacts1} and the fact that $w \gtrsim 1$.
For (b), observe that by part (a),
$P$ maps $\Esp$ into $W(L^\infty, \Esp)$
and coincides with the identity operator on $\SEsp$.
\begin{rem}
\label{rem_setting_unif}
The estimates in Proposition \ref{prop_P_into_am} hold uniformly for all the
spaces $\Esp$ with the same weight $w$ and the same constant $\consEw$ (cf.
Equation \eqref{weight_w_admissible}).
\end{rem}
In the applications the same projection $P$ will be used with different
spaces $\Esp$ and corresponding subspaces $\SEsp$, providing a unified
analysis
of a whole class of functional spaces. This is why Remark
\ref{rem_setting_unif}
is relevant.

\end{proof}
\section{Approximation of phase-space projections}
\label{sec_desc_proy}
In this section we prove the main technical estimate of the article.
Given the setting from Section \ref{sec_model} and
a partition of unity $\sum_\gamma \eta_\gamma \equiv 1$,
we will show that the phase-space projection $P(f)$ from Section
\ref{sec_model} can be resynthesized from the
phase-space localized pieces
$\sett{P(f \eta_\gamma)}_\gamma$.
Note that $P(f)$ can be trivially recovered from
$\sett{P(f\eta_\gamma)}_\gamma$ by simply summing all these
functions.
We will prove that this reconstruction can also be achieved by
placing the localized pieces on top of the (morally) corresponding
regions of the phase-space. This controlled synthesis will then
allow us to quantify the relation between
$P(f)$ and $\sett{P(f \eta_\gamma)}_\gamma$ and yield the
main result on the characterization of the norm of $\SEsp$.
\subsection{Setting}
Let us first formally introduce all the required ingredients.
Suppose that Assumptions (A1) and (A2) from Section \ref{sec_model}
hold. We now state
Assumption (B1) introducing the partition of unity
covering phase-space and the norm used to measure it.
\begin{itemize}
\item[(B1)]
\begin{itemize}\setlength{\itemsep}{+2mm}
\item $\Gamma \subseteq \Gc$ is a relatively separated set.
\item $\set{\eta_\gamma}{\gamma \in \Gamma}$ is a set of
$\wweak$-molecules enveloped by a function $g$. More precisely,
\begin{itemize}\setlength{\itemsep}{+2mm}
\item $\abs{\eta_\gamma(x)} \leq g(\gamma^{-1}x),
\qquad (x \in \Gc, \gamma \in \Gamma)$,
\item $g \in \wweak$.
\end{itemize}
\item $\sett{\eta_\gamma}_\gamma$ is a bounded partition
of unity. That is, 
\[
\sum_\gamma \eta_\gamma \equiv 1,
\quad \mbox{ and } \quad
\sum_\gamma \abs{\eta_\gamma} \in L^\infty(\Gc).
\]
\item $\Bsp$ is a solid, isometrically left-translation invariant
Banach space such that $W(L^\infty,L^1_w) \hookrightarrow \Bsp$.
\end{itemize}
\end{itemize}
\begin{rem}
\label{rem_B_into}
By Lemma \ref{lemma_amfacts1}, $\Bsp \into W(L^1,L^\infty)$.
In addition, by the definition of translation invariant space
$L^1*\Bsp \hookrightarrow \Bsp$.
\end{rem}
\begin{rem}
Note that the conditions in (B1) allow for the usual bounded uniform
partitions of unity \cite{fe83, fegr85, fe87} - where functions are supported
on a family of compact sets having a bounded number of overlaps - but also
for functions having non-compact support.
\end{rem}
\subsection{Vector-valued analysis and synthesis}
Let us now describe the operators mapping a function $f$
into its phase-space localized pieces, by means of the partition
of unity $\sett{\eta_\gamma}_\gamma$. Let the \emph{analysis operator}
$\coefv$ be formally defined by,
\begin{align}
\label{eq_coefv}
\coefv(f) := \left( P(f \eta_\gamma) \right)_{\gamma \in \Gamma}.
\end{align}
For each $U$, a relatively compact neighborhood of the identity in $\Gc$,
we also formally define the \emph{synthesis operator}
$\recv_U$, acting on a sequence of functions by,
\begin{align}
\label{eq_recv}
 \recv_U(\left( f_\gamma \right)_{\gamma \in \Gamma})
:= \sum_\gamma P(f_\gamma) \chi_{\gamma U}.
\end{align}
The operator $\recv_U$ will be used as an approximate left-inverse
of the vector valued analysis operator $\coefv$. Let us now establish the
mapping properties of these operators.
\begin{prop}
\label{prop_vector_synt}
Under Assumptions (A1), (A2) and (B1) the following statements hold.
\begin{itemize}
 \item[(a)]
The analysis operator $\coefv$ maps $W(L^\infty,\Esp)$ boundedly into
$\EspdB(\Gamma)$.
In particular (cf. Proposition \ref{prop_P_into_am}) it maps 
$\SEsp$ boundedly into $\EspdB(\Gamma)$.
\item[(b)]
For every relatively compact neighborhood of the identity $U$,
and every sequence $F \in \EspdB$,
the series defining $\recv_U(F)$ converge absolutely in the norm of $\Bsp$ at
every point.
Moreover, the synthesis operator $\recv_U$ maps 
$\EspdB(\Gamma)$ boundedly into $\Esp$
(with a bound that depends on U). 
\end{itemize}
\end{prop}
\begin{proof}
To prove (a) let $f \in W(L^\infty,\Esp)$.
Since $\eta_\gamma$ is bounded, $f \eta_\gamma \in
W(L^\infty,\Esp) \subseteq \Esp$.
By the pointwise bound for $P$ (cf. Equation \eqref{eq_P_dominated}),
\begin{align*}
\abs{P(f \eta_\gamma)(x)} &\leq
\int_\Gc \abs{f(y)} g(\gamma^{-1}y) H(y^{-1}x) dy
\\
&= \left( \abs{f}L_\gamma g \right) * H (x).
\end{align*}
Since $\Bsp$ is solid and $L^1*\Bsp \hookrightarrow \Bsp$,
we have,
\begin{align*}
\norm{P(f \eta_\gamma)}_\Bsp \leq
\norm{H}_\Bsp
\int_\Gc \abs{f(y)} g(\gamma^{-1}y) dy
\lesssim
\norm{H}_{W(L^\infty,L^1_w)}
\int_\Gc \abs{f(y)} g(\gamma^{-1}y) dy.
\end{align*}
Now the solidity of $\Esp$ and Lemma \ref{lemma_amfacts3} yield,
\begin{align*}
 \norm{\coefv(f)}_{\EspdB}
 \lesssim
\norm{f}_{W(L^\infty,\Esp)} \norm{g}_\wweak.
\end{align*}
To prove (b) consider a family
$F \equiv (f_\gamma)_\gamma \in \EspdB$.
For each $\gamma \in \Gamma$, $f_\gamma \in \Bsp \subseteq
W(L^1,L^\infty)$, so by Proposition \ref{prop_P_into_am},
$P(f_\gamma)$ is well-defined and satisfies,
\begin{align*}
\abs{P(f_\gamma)(x)}
\lesssim
\norm{f_\gamma}_{W(L^1,L^\infty)} \norm{H}_{W_R(L^\infty,L^1_w)}
\lesssim
\norm{f_\gamma}_\Bsp \norm{H}_{W_R(L^\infty,L^1_w)}.
\end{align*}
Hence, for every $x \in \Gc$,
$\abs{\recv(F)(x)}
\lesssim
\sum_\gamma \norm{f_\gamma}_\Bsp \chi_U(\gamma^{-1}x)$.
Since $U$ is relatively compact, $\chi_U \in W_R(L^\infty,L^1_w)$
and consequently Lemma \ref{lemma_amfacts2} together with
the solidity of $\Esp$ imply that
$\norm{\recv_U(F)}_\Esp \lesssim 
\norm{\chi_U}_{W_R(L^\infty,L^1_w)}
\norm{F}_{\EspdB}$.
\end{proof}
\subsection{Approximation of the projector}
Now we can state the main result on the approximation of $P$.
For every $U$, relatively compact neighborhood of the identity in $\Gc$,
consider the approximate projector $P_U: W(L^\infty,\Esp) \to \Esp$ given by,
\begin{align}
\label{deffPU}
 P_U(f) := \sum_{\gamma \in \Gamma} P(f \eta_\gamma) \chi_{\gamma U}.
\end{align}
Since $P_U = \recv_U \circ \coefv$, $P_U$ is well-defined.
We will prove that $P_U$ approximates $P$ in the following way.
\begin{theorem}
\label{th_approx}
Given $\varepsilon >0$, there exists $U_0$, a relatively compact neighborhood
of e such that for all $U \supseteq U_0$,
\begin{align*}
 \norm{P(f) - P_U(f)}_\Esp \leq \varepsilon \norm{f}_{W(L^\infty,\Esp)},
\quad (f \in W(L^\infty,\Esp)).
\end{align*}
\end{theorem}
\begin{rem}
The neighborhood $U_0$ can be chosen uniformly for any class of spaces
$\Esp$ having the same weight $w$ and the same constant $\consEw$ (cf.
Equation \eqref{weight_w_admissible}).

Concerning the ingredients introduced in Assumptions (A2) and (B1),
the choice of $U_0$ only depends on
$\norm{H}_{W(L^\infty,L^1_w)}$,
$\norm{H}_{W_R(L^\infty,L^1_w)}$,
$\norm{g}_{\wweak}$ and
$\rel(\Gamma)$ (cf. Equation \eqref{eq_spread}).
\end{rem}

In order to prove Theorem \ref{th_approx} we introduce the following
auxiliary function.
For each $U$, let
$G_U:\Gc \to [0,+\infty)$ be defined by,
\begin{align}
\label{deffGU}
G_U(x) := \sup_{y \in \Gc}
\sum_{\gamma \in \Gamma} (g*\chi_V)(\gamma^{-1}y)
\chi_{\gamma (\Gc \setminus U)}(yx).
\end{align}
Observe that $G_U$ is defined as a supremum of a family of sums.
The estimates for $P$ that we will derive in terms of $G_U$ are different
from the usual convolution estimates involving Wiener amalgam norms of $g$
and will be crucial for the proof of Theorem \ref{th_approx}.
Before proving that theorem we establish some necessary
estimates for the auxiliary function.

\begin{lemma}
\label{lemmaGU}
The function $G_U$ satisfies $\norm{G_U}_{L^\infty(\Gc)} \lesssim 1$
(with a bound independent of $U$).
Moreover, for every compact set $K \subseteq \Gc$,
\begin{align*}
 \norm{G_{U \pointdot K}}_{L^\infty(K)}
 \lesssim
 \int_{V \pointdot (\Gc \setminus U)} (g*\chi_V)_{\#}(x) w(x) dx.
\end{align*}
\end{lemma}
\begin{proof}
Let a compact set $K$ and an element $x \in K$ be given.
For $y \in \Gc$, if  $yx \in \gamma (\Gc \setminus (UK))$, then
$\gamma^{-1}yx \notin UK$, so $\gamma^{-1}y \notin U$.

Therefore, 
\begin{align*}
&\sum_\gamma (g*\chi_V)(\gamma^{-1}y)
\chi_{\gamma (\Gc \setminus (UK))}(yx)
\leq
\sum_{\gamma: \gamma^{-1}y \notin U} (g*\chi_V)(\gamma^{-1}y)
\\
&\qquad
\lesssim
\sum_{\gamma: \gamma^{-1}y \notin U}
\int_\Gc (g*\chi_V)_{\#}(t^{-1}\gamma^{-1}y) \chi_V(t) dt
\\
&\qquad =
\int_\Gc (g*\chi_V)_{\#}(t^{-1})
\sum_{\gamma: \gamma^{-1}y \notin U} \chi_V(\gamma^{-1}yt) dt.
\end{align*}
Since $\Gamma$ is relatively separated,
$\sum_\gamma \chi_V(\gamma^{-1}yt) 
= \sum_\gamma \chi_V(t^{-1}y^{-1}\gamma) \lesssim 1$.
In addition, if $\gamma^{-1}yt \in V$ and $\gamma^{-1}y \notin U$
then $t=(\gamma^{-1}y)^{-1} \gamma^{-1}yt \in (\Gc\setminus U)^{-1} V$.

Hence,
\begin{align*}
G_{U \pointdot K}(x) \lesssim
\int_{(\Gc\setminus U)^{-1} \pointdot V} (g*\chi_V)_{\#}(t^{-1}) dt
=
\int_{V \pointdot (\Gc\setminus U)} (g*\chi_V)_{\#}(t) \Delta(t^{-1}) dt.
\end{align*}
Since $\Delta(t^{-1}) \lesssim \Delta(t^{-1}) w(t^{-1}) = w(t)$
the desired bound follows.
Reexamining the computations above we see that,
$\norm{G_U}_{L^\infty(\Gc)}
\lesssim \int_\Gc (g*\chi_V)_{\#}(t) w(t) dt$.
Since $g \in \wweak$, the last integral is finite
and we get the desired uniform bound.
\end{proof}
Now we can prove Theorem \ref{th_approx}.
\begin{proof}[Proof of Theorem \ref{th_approx}]
Let $f \in W(L^\infty,\Esp)$
and let $U$ be a relatively compact neighborhood of e.
Since $\sum_\gamma \eta_\gamma \equiv 1$,
\begin{align*}
P(f) - P_U(f) =
\sum_\gamma P(f\eta_\gamma) -
\sum_\gamma P(f\eta_\gamma) \chi_{\gamma U}
= \sum_\gamma P(f\eta_\gamma) \chi_{\gamma(\Gc \setminus U)}.
\end{align*}
Consequently, by the pointwise bound for $P$ (cf. Equation
\eqref{eq_P_dominated}), for $x \in \Gc$,
\begin{align*}
 \abs{P(f)(x) - P_U(f)(x)} &\leq
\sum_\gamma \int_{\Gc} \abs{f(y)} g(\gamma^{-1}y) H(y^{-1}x)
\chi_{\gamma(\Gc \setminus U)}(x) dy.
\end{align*}
Since $\abs{f(y)} \lesssim \int f^{\#}(z) \chi_V(y^{-1}z) dz$,
we have,
\begin{align*}
 \abs{P(f)(x) - P_U(f)(x)} &\lesssim
\int_{\Gc} f^{\#}(z) \sum_\gamma
\int_{\Gc} \chi_V(y^{-1}z) g(\gamma^{-1}y) H(y^{-1}x)
\chi_{\gamma(\Gc \setminus U)}(x) dy dz.
\end{align*}
Observe that if $y^{-1}z \in V$,
then $y^{-1}x=y^{-1}z z^{-1}x \in V (z^{-1}x)$,
and therefore $H(y^{-1}x) \leq H_{\#}(z^{-1}x)$. Hence,
\begin{align*}
 \abs{P(f)(x) - P_U(f)(x)} &\lesssim
\int_{\Gc} f^{\#}(z) H_{\#}(z^{-1}x) \sum_\gamma
\int_{\Gc} g(\gamma^{-1}y) \chi_V(y^{-1}z)
\chi_{\gamma(\Gc \setminus U)}(x) dy dz
\\
&=
\int_{\Gc} f^{\#}(z) H_{\#}(z^{-1}x) \sum_\gamma
(g*\chi_V)(\gamma^{-1}z) \chi_{\gamma(\Gc \setminus U)}(x) dz
\\
&\leq
\int_{\Gc} f^{\#}(z) H_{\#}(z^{-1}x) G_U(z^{-1}x) dz
\\
&=
f^{\#} * \left(H_{\#}G_U\right)(x).
\end{align*}
Consequently,
\begin{align*}
\norm{P(f)-P_U(f)}_\Esp
\lesssim \norm{f^{\#}}_\Esp \norm{H_{\#} G_U}_{L^1_w}
= \norm{f}_{W(L^\infty,\Esp)} \norm{H_{\#} G_U}_{L^1_w}.
\end{align*}
Therefore, it suffices to show that
$\norm{H_{\#} G_U}_{L^1_w} \longrightarrow 0$, as $U$ grows to $\Gc$.
For every compact set $K \subseteq \Gc$, Lemma \ref{lemmaGU} implies that
\begin{align*}
&\int_\Gc H_{\#}(z)G_U(z) w(z) dz
\\
&\qquad
\leq 
\norm{G_U}_{L^\infty(K)} \norm{H_{\#}}_{L^1_w}
+
\norm{G_U}_{L^\infty(\Gc)}
\int_{\Gc \setminus K} H_{\#}(z) w(z) dz
\\
&\qquad
\lesssim
\norm{G_U}_{L^\infty(K)}
+
\int_{\Gc \setminus K} H_{\#}(z) w(z) dz.
\end{align*}
Given $\varepsilon>0$, we choose
a compact set containing the identity $K$
such that the second term in the last inequality
is less that $\varepsilon$.
Since $g \in \wweak$,
we can also choose a compact set containing the identity
$Q \subseteq \Gc$ such that
\begin{align*}
\int_{\Gc \setminus Q} (g*\chi_V)_{\#}(x) w(x) dx < \varepsilon. 
\end{align*}
Set $U_0 := V Q K$. If $U \supseteq U_0$ is a relatively compact
neighborhood of $e$, then, using Lemma \ref{lemmaGU},
\begin{align*}
\norm{G_U(z)}_{L^\infty(K)}
\leq \norm{G_{VQK}}_{L^\infty(K)}
\lesssim
\int_{V(\Gc \setminus (VQ))} (g*\chi_V)_{\#}(x) w(x) dx.
\end{align*}
Since $V=V^{-1}$,
we have that $V(\Gc \setminus (VQ)) \subseteq (\Gc \setminus Q)$
and consequently
$\norm{G_U(z)}_{L^\infty(K)} \lesssim \varepsilon$.
Hence, we have shown that for $U \supseteq U_0$,
$\norm{H_{\#} G_U}_{L^1_w} \lesssim \varepsilon$. This completes the proof.
\end{proof}
\section{Approximation of phase-space multipliers}
\label{sec_approx_mult}
We will now interpret Theorem \ref{th_approx} as a result
about approximation of \emph{phase-space multipliers}.
Let us suppose that Assumptions (A1), (A2) and (B1) hold.

For $m \in L^\infty(\Gc)$, the \emph{multiplier} $\multm:\SEsp \to \SEsp$
with \emph{symbol} $m$ is defined by,
\begin{align}
\label{deffMm}
 \multm(f) := P(mf),
\qquad (f \in \SEsp).
\end{align}
The operator $\multm$ is clearly bounded by Proposition \ref{prop_P_into_am}
and
the solidity of $\Esp$. When the space $\SEsp$ is taken to be the range
of the abstract wavelet transform associated with an unitary
representation of $\Gc$, these operators are called localization operators or
wavelet multipliers (see for example \cite{hewo96, wong02, limowo08}). 
(More precisely, the operators $\multm$ are unitary
equivalent to localization operators, see Section \ref{sec_app_co} for
further details). When $\SEsp$ is the range of the Short-time Fourier
transform the corresponding operators are known as STFT multipliers or
Time-Frequency localization operators (\cite{da88, cogr03, bo04-2, cogr06}).

Using the approximation of the projector from the previous section,
we construct an approximation of the multiplier $\multm$.
For a relatively compact neighborhood of the identity $U$,
let $\multmU: \SEsp \to \SEsp$ be defined by,
\begin{align*}
 \multmU(f) := PP_U(mf),
\qquad (f \in \SEsp).
\end{align*}

Now Theorem \ref{th_approx} implies the following.
\begin{theorem}
\label{th_approx_mult}
For each $m \in L^\infty(\Gc)$,
$\multmU \longrightarrow \multm$ in operator norm,
as $U$ ranges over the class of relatively compact neighborhoods of the
identity, ordered by inclusion.
Moreover, convergence is uniform on any bounded class of symbols.
\end{theorem}
\begin{proof}
By Proposition \ref{prop_P_into_am}, for $f \in \SEsp$,
\begin{align*}
\norm{\multmU(f) - \multm(f)}_\Esp
= \norm{PP_U(mf)-PP(mf)}_\Esp
\lesssim
\norm{P_U(mf)-P(mf)}_\Esp.
\end{align*}
By Theorem \ref{th_approx},
$\norm{P_U(mf)-P(mf)}_\Esp \lesssim \delta(U) \norm{mf}_{W(L^\infty,\Esp)}$,
for some function $\delta$ such that $\delta(U) \longrightarrow 0$, as $U$
grows to $\Gc$.
Finally, since $m \in L^\infty(\Gc)$ and $f \in \SEsp$, 
the embedding $\SEsp \into W(L^\infty,\Esp)$ in Proposition
\ref{prop_P_into_am} implies that
$\norm{mf}_{W(L^\infty,\Esp)}
\lesssim 
\norm{f}_{W(L^\infty,\Esp)}
\lesssim
\norm{f}_\Esp$, and the conclusion follows. Observe that if $m$ belongs
to a certain bounded subset of $L^\infty$, then the last estimate holds
uniformly on that set.
\end{proof}

\section{Characterization of the atomic space with multipliers}
\label{sec_main_th}
We can finally prove the main abstract result on
the characterization of the atomic space with phase-space multipliers.
\begin{theorem}
\label{th_coverings_do_cover}
Under Assumptions (A1), (A2) and (B1), the map
\begin{align*}
\coefv: \SEsp &\to \EspdB
\\
f &\mapsto (P(f \eta_\gamma))_\gamma
\end{align*}
is left-invertible. Consequently, the following norm equivalence
holds for $f \in \SEsp$,
\begin{align*}
 \norm{f}_\Esp \approx
\norm{
({\norm{P(f\eta_\gamma)}_\Bsp})_\gamma
}_{\Espd}.
\end{align*}
\end{theorem}
\begin{rem}
The fact that there is such a liberty to choose the BF space $\Bsp$
is analogous to the fact that for coorbit spaces
only the ``global behavior'' of the norm imposed on the wavelet transform
matters. See \cite[Theorem 8.3]{fegr89-1}.
\end{rem}
\begin{rem}
The norm equivalence holds uniformly for any class of spaces
$\Esp$ having the same weight $w$ and the same constant $\consEw$ (cf.
Equation \eqref{weight_w_admissible}).
\end{rem}

\begin{proof}
With the notation of Section \ref{sec_approx_mult},
using Theorem \ref{th_approx_mult} with symbol $m \equiv 1$,
we choose a relatively compact neighborhood of the identity $U$
such that $\multoneU$ is invertible.
Since the operator $P_U$ (cf. Equation \eqref{deffPU})
can be factored as $P_U=\recv_U \coefv$, we have that, $\multoneU=P \recv_U
\coefv$. Since  $\multoneU$ is invertible, $\coefv$ is left-invertible, as
claimed.
This implies
that $\norm{f}_\Esp \lesssim \norm{\coefv(f)}_{\EspdB}$,
for $f \in \SEsp$. The converse inequality is just the
boundedness of $\coefv$ and was proved in Proposition
\ref{prop_vector_synt}.
\end{proof}
\section{The case of atomic decompositions}
\label{sec_setting_atdesc}
We now consider a setting where the atomic space 
from Section \ref{sec_model} has a distinguished atomic decomposition.
We prove a number of technical results that will allow us to
finely adjust the general results of Section \ref{sec_main_th}
in order to get sharper statements for certain applications.

It is known that under very general conditions
any instance of the model introduced in Section \ref{sec_model}
has an associated atomic decomposition (see \cite{nasu10}), but
nevertheless some matters naturally pertain to the general setting
while others are specific to the case of atomic decompositions.

Let us recall Assumption (A1) from Section \ref{sec_model}.
\begin{itemize}
\item[(A1)]
\begin{itemize}
\item  $\Esp$ is a solid, translation invariant BF space,
called \emph{the environment}.
\item $w$ is an admissible weight for $\Esp$.
\item $\SEsp$ is a closed complemented subspace of $\Esp$,
called \emph{the atomic subspace}.
\end{itemize}
\end{itemize}
We now state Assumption (A2') introducing new ingredients
to the model.
\begin{itemize}
\item[(A2')]
\begin{itemize}\setlength{\itemsep}{+2mm}
\item $\Lambda \subseteq \Gc$ is a relatively separated set.
Its points will be called \emph{nodes}.
\item $\set{\varphi_\lambda}{\lambda \in \Lambda}$ and
$\set{\psi_\lambda}{\lambda \in \Lambda}$ are sets of
$\wstrong$ molecules, enveloped by a function $h$. That is,
\begin{itemize}\setlength{\itemsep}{+2mm}
\item $\abs{\varphi_\lambda(x)}, \abs{\psi_\lambda(x)}
\leq h(\lambda^{-1}x),
\qquad (x \in \Gc, \lambda \in \Lambda),$
\item $h \in \wstrong$.
\end{itemize}
The sets $\sett{\varphi_\lambda}_\lambda$ and
$\sett{\psi_\lambda}_\lambda$ will be called \emph{atoms}
and \emph{dual atoms} respectively.

\item  $\SEsp \subseteq \Esp$ has the following atomic decomposition.
\begin{itemize}
 \item[(a)] For every $c \in \Espd(\Lambda)$,
the series $\sum_\lambda c_\lambda \varphi_\lambda$ belong to $\SEsp$.
\footnote{The convergence of the series is clarified in Lemma
\ref{lemma_amfacts2}.}
 \item[(b)] For all $f \in \SEsp$, the following expansion holds,
\begin{align}
\label{at_desc}
 f = \sum_{\lambda \in \Lambda} \ip{f}{\psi_\lambda} \varphi_\lambda.
\end{align}
\end{itemize}
 \end{itemize}
\end{itemize}
Associated with the atoms we consider
the \emph{analysis} and \emph{synthesis} operators given by,
\begin{align}
\label{eq_coef}
&\coef:\Esp \to \Espd,
\quad
\coef(f) := ( \ip{f}{\psi_\lambda} )_\lambda,
\\
\label{eq_rec}
&\rec:\Espd \to \Esp,
\quad
\rec(c) := \sum_\lambda c_\lambda \varphi_\lambda.
\end{align}
We also consider their \emph{formal adjoints} given by,
\begin{align}
\label{eq_coefp}
&\coefp:\Espd \to \Esp,
\quad
\coefp(c) := \sum_\lambda c_\lambda \psi_\lambda,
\\
\label{eq_recp}
&\recp:\Esp \to \Espd,
\quad
\recp(f) := ( \ip{f}{\varphi_\lambda} )_\lambda.
\end{align}
\begin{rem}
\label{rem_coefrec_bounded}
Under Assumptions (A1) and (A2'), 
the operators $\coef, \rec, \coefp, \recp$
are well-defined and bounded
by Lemma \ref{lemma_amfacts2} and the fact that
$h \in \wstrong \subseteq W_R(L^\infty,L^1_w)$.
\end{rem}
We also consider the operator $P:\Esp \to \SEsp$
defined by $P := \rec \circ \coef$. Hence,
\begin{align}
\label{deffP}
P(f) = \sum_{\lambda \in \Lambda}
\ip{f}{\psi_\lambda} \varphi_\lambda.
\end{align}
According to (A2'), $P$ is a projector from $\Esp$ onto $\SEsp$.

We will now see that the setting introduced by (A1) and (A2')
can be regarded as an instance of the one set by (A1) and (A2).
We first introduce the function $H$ required by (A2).
Let $H: \Gc \to [0,+\infty)$ be defined by
\begin{align}
\label{deffH}
 H(x) := \sup_{y \in \Gc}
\sum_{\lambda \in \Lambda} h(\lambda^{-1}y) h(\lambda^{-1}yx).
\end{align}
The following lemma shows that $P$ and $H$ satisfy the conditions in (A2).
\begin{lemma}
\label{lemmaH}
\label{lemmaPH}
Under Assumptions (A1) and (A2') the following statements hold.
\begin{itemize}
\item[(a)] The function $H$ (cf. Equation \eqref{deffH}) belongs both to 
$W(L^\infty, L^1_w)$ and $W_R(L^\infty, L^1_w)$.

\item[(b)] For every $f \in W(L^1,L^\infty_{1/w})$, the function
$P(f) = \sum_\lambda \ip{f}{\psi_\lambda}
\varphi_\lambda$ is well-defined
(with absolute convergence at every point)
and satisfies the following pointwise estimate,
\begin{align*}
\abs{P(f)(x)} \leq \int_{\Gc} \abs{f(y)} H(y^{-1}x) dy,
\qquad (x \in \Gc).
\end{align*}
Moreover,
$\norm{P(f)}_{L^\infty_{1/w}} \lesssim
\norm{f}_{W(L^1,L^\infty_{1/w})} \norm{H}_{W_R(L^\infty,L^1_w)}$.
\end{itemize}
\end{lemma}
\begin{proof}
Part (a) follows from a straightforward computation.
One can first establish the estimates,
\begin{align*}
H^{\#}(x) \lesssim 
\int_\Gc h_{\#}(t^{-1}) (h_{\#})^{\#} (t^{-1}x) dt,
\quad \mbox{and} \quad
H_{\#}(x) \lesssim 
\int_\Gc (h_{\#})^{\#}(t^{-1}) (h_{\#}) (t^{-1}x) dt,
\end{align*}
and then deduce that
$\norm{H}_{W(L^\infty, L^1_w)} + \norm{H}_{W_R(L^\infty, L^1_w)}
\lesssim \norm{h}_\wstrong^2$.

Using the enveloping condition in (A2') we get
the desired pointwise estimate for $P$.
Part (b) then follows from part (a)
and Lemma \ref{lemma_amfacts1}.
\end{proof}
\subsection{Weak continuity of the atomic decomposition of $\SEsp$}
Suppose that Assumptions (A1) and (A2') hold.
Lemmas \ref{lemma_amfacts1} and \ref{lemma_amfacts2}
give the embeddings $\Esp \hookrightarrow W(L^1,L^\infty_{1/w})$
and $\Espd \hookrightarrow \ell^\infty_{1/w}$.
We denote by $(\Espd,\ell^1_w)$ the space $\Espd$ considered
with the restriction of the weak* star topology of $\ell^\infty_{1/w}$.
Likewise, since by Lemma \ref{lemma_amfacts1},
$W(L^1,L^\infty_{1/w})$ embeds into the dual space of
$W(L^\infty, L^1_w)$, we let $(\Esp, W(L^\infty, L^1_w))$
stand for space $\Esp$ considered with the topology induced by
the linear functionals obtained by integration against $W(L^\infty, L^1_w)$
functions. Observe that, since this family of functionals separates points,
the corresponding topology is Hausdorff.

We will now establish the continuity of the 
maps that implement the atomic decomposition of $\SEsp$
with respect to these coarser topologies. This will allow us to use density
arguments for $\SEsp$. This is irrelevant when the atomic
decomposition in Equation \eqref{at_desc} converges in the norm of $\Esp$,
but is important to make the abstract results fully applicable.
\begin{prop}
\label{prop_weak_cont}
Under Assumptions (A1) and (A2') the following statements hold.
\begin{itemize}
\item[(a)] The map $\coef: (\Esp, W(L^\infty, L^1_w)) \to (\Espd,\ell^1_w)$
is continuous.
\item[(b)] For $c \in \Espd$, the series defining $\rec(c)$
converge unconditionally in the $(\Esp, W(L^\infty, L^1_w))$ topology.
Moreover, the map $\rec: (\Espd,\ell^1_w) \to (\Esp, W(L^\infty, L^1_w))$
is continuous.
\end{itemize}
Similar statements hold for the operators $\coefp$ and $\recp$
(cf. Equations \eqref{eq_coefp} and \eqref{eq_recp}).
\end{prop}
\begin{proof}
The operators $\coef, \rec, \coefp, \recp$
are formally related by,
\begin{align*}
\ip{\coef(f)}{c}=\ip{f}{\coef'{c}},
\\
\ip{\rec(c)}{f}=\ip{c}{\rec'{f}},
\end{align*}
with $f  \in \Esp$ and $c \in \Espd$.
The proposition follows easily from here.
All the technical details on interchange
of summation and integration
can be justified using
Lemmas \ref{lemma_amfacts1} and \ref{lemma_amfacts2}.
\end{proof}

\section{More general partitions of unity}
\label{sec_more_general}
Under additional assumptions we can extend Theorem
\ref{th_coverings_do_cover}
to the case where the condition on the partition of unity:
$\sum_\gamma \eta_\gamma \equiv 1$
is relaxed to: \mbox{$0 < A \leq \sum_\gamma \eta_\gamma \leq B < \infty$}.
To avoid altering the ongoing notation we keep the assumption that
$\sum_\gamma \eta_\gamma \equiv 1$ and introduce a new (generalized)
partition of unity $\sett{\theta_\gamma: \gamma \in \Gamma}$ related
to the one in Assumption (B1) by $\theta_\gamma = m \eta_\gamma$,
where $0 < A \leq m \leq B < \infty$.
This is the general form of a family of functions 
$\sett{\theta_\gamma: \gamma \in \Gamma}$ enveloped by $g$ and
whose sum is nonnegative and bounded away from zero and infinity.

Consider the setting of Section \ref{sec_setting_atdesc}.
The problem of extending Theorem \ref{th_coverings_do_cover}
to this new partition of unity can be reduced to the one
of establishing the invertibility of the multiplier $\multm$ (cf. 
Equation \eqref{deffMm}).
To this end, we will extend the atomic decomposition 
on Equation \eqref{at_desc} to an adequate Hilbert space $\Hsp$,
then prove the invertibility of $\multm$ on $\Hsp$ and finally
use the spectral invariance of a certain subalgebra 
of the algebra of bounded operators on $\ell^2$ to
deduce the invertibility of $\multm$ on $\SEsp$. This is where
certain restrictions on the geometry of $\Gc$ need to be imposed.
For the case of time-frequency decompositions and modulation spaces,
this line of reasoning is hinted on the final remark of \cite{cogr06} and
developed for a very general class of symbols and weighted modulation
spaces in \cite{grto10}.

Suppose that Assumptions (A1) and (A2') hold (cf. Section
\ref{sec_setting_atdesc}). We will now introduce Assumptions (C1) and (C2)
and present the extension of Theorem \ref{th_coverings_do_cover}.
\subsection{Assumption (C1)}
We will use a key result from from \cite{fegrle08}. To this end we introduce
the following conditions for a discrete group $\Omega$ and a weight $u$ on
it.
\begin{definition}
We say that the pair $(\Omega,u)$ satisfies the FGL-conditions
if the following holds.
\begin{itemize}
\item $\Omega$ is a discrete, amenable, rigidly symmetric group.
\item $u: \Omega \to [1,\infty)$ is a submultiplicative, symmetric
weight that satisfies $u(e)=1$ and,
\begin{align*}
&\lim_{n \rightarrow +\infty} \sup_{x \in U^n} u(x)^{1/n} = 1,
\mbox{ and,}
\\
&\inf_{x \in U^n \setminus U^{n-1}} u(x) \approx
\sup_{x \in U^n \setminus U^{n-1}} u(x),
\quad (n \in \Nst),
\end{align*}
for some symmetric generating subset $U$ of $\Omega$, containing the identity
element.
\end{itemize}
\end{definition}
For an explanation of the FGL-conditions and their relation to
other notions for groups (such as polynomial growth) see
\cite{fegrle08, fegrle06} and the references therein. 
In Proposition \ref{prop_suf_c1} we give more concrete sufficient
conditions for the applicability of the FGL-conditions to our setting.

Now we introduce the following assumption on the geometry
of $\Gc$ and the set of nodes $\Lambda$ that provides the atomic
decomposition of $\SEsp$. This condition will be satisfied in the
applications to
time-frequency analysis but not in the case of time-scale decompositions.

\begin{itemize}
\item[(C1)] We assume the following.
\begin{itemize}
\item $\Gc$ is an IN group.\footnote{Remember that, by
convention, we also assume that the distinguished neighborhood $V$ is
invariant under inner automorphisms.}
\item The set $\Lambda$ is a closed, discrete subgroup of $\Gc$ that,
considered as a topological group in itself, satisfies the FGL-conditions
with respect to the restriction of the weight $w$.
\end{itemize}
\end{itemize}
\begin{rem}
\label{rem_lp_are_mod}
The fact that $\Gc$ is an IN group implies that it is unimodular
(i.e. $\Delta \equiv 1$) (see \cite{palmer78}). As
a consequence, the weight $w$ is symmetric (i.e. $w(x)=w(x^{-1})$).

The submultiplicativity of $w$ now implies
that $(1/w)(xy) \leq w(x) (1/w)(y)$. This equation in turn implies
that the weight $w$ is admissible for all the spaces $L^p_w$
and $L^p_{1/w}$, $(1 \leq p \leq +\infty)$.
\end{rem}

Since under Assumption (C1) $\Lambda$ is a subgroup, it is possible to
consider convolution operators on $\Espd(\Lambda)$. The space
$\Espd(\Lambda)$ is always left-invariant, but for a general group $\Gc$ it
may not be right-invariant (even if $\Esp$ is). Using the fact that in (C1)
$\Gc$ is assumed to be an IN group, the following proposition can be easily
proved.

\begin{prop}
\label{Ed_is_right_invariant}
Under Assumption (C1),
$\Espd(\Lambda) * \ell^1_w(\Lambda) \subseteq \Espd(\Lambda)$,
with the corresponding norm estimate.
\end{prop}

Before introducing the second assumption we give some sufficient conditions
for Assumption (C1) to hold. Recall that a group is called almost connected
if the quotient by the connected component of the identity element is
compact.
\begin{prop}
\label{prop_suf_c1}
Suppose that Assumptions (A1), (A2') and (B1) hold
and that, in addition, $\Gc$ is an almost connected IN group.
Suppose that $\Lambda$ is a discrete, closed, finitely-generated
subgroup of $\Gc$
and that the weight $w$ satisfies $w(e)=1$,
the Gelfand-Raikov-Shilov condition,
\begin{align}
\label{eq_grs}
\lim_{n \rightarrow +\infty} w(\lambda^n)^{1/n}=1,
\mbox{ for all $\lambda \in \Lambda$},
\end{align}
and the condition,
\begin{align*}
\inf_{x \in U^n \setminus U^{n-1}} w(x) \approx
\sup_{x \in U^n \setminus U^{n-1}} w(x)
\mbox{, for all }n \in \Nst,
\end{align*}
for some symmetric generating subset $U$ of $\Lambda$, containing the
identity.

Then, the conditions in (C1) are satisfied.
\end{prop}
\begin{proof}
The group $\Gc$ is an almost connected IN group and therefore has
polynomial growth (see \cite{palmer78}). Since $\Lambda$ is discrete
and closed in $\Gc$ it also has polynomial growth (with respect
to the counting measure). Indeed, using the fact that 
$\Lambda$ is discrete and closed it follows that there exist
$W$, a compact neighborhood of the identity in $\Gc$, such that
$\lambda W \cap \lambda' W = \emptyset$ for
any two distinct elements $\lambda, \lambda' \in \Lambda$.
Then, for any finite set $F \subseteq \Lambda$, the cardinality
of $F^n$ is dominated by $\mes{(FW)^n}$.

Hence, $\Lambda$ is a finitely-generated discrete group of polynomial growth.
Therefore $\Lambda$ is amenable (see \cite{palmer78}). In addition, by
Gromov's structure theorem \cite{gromov81},
$\Lambda$ has a nilpotent subgroup of finite index.
Corollary 3 from \cite{leptinpoguntke79} implies
that $\Lambda$ is rigidly symmetric (see also \cite{fegrle08}).
Finally, since $\Lambda$ is a finitely-generated discrete group of polynomial
growth, Theorem 1.3 from \cite{fegrle06} implies that the GRS condition
in Equation \eqref{eq_grs} implies the condition required in (C1).
\end{proof}
\subsection{Assumption (C2)}
In order to introduce the second assumption, suppose that Assumptions
(A1), (A2') and (C1) hold and let $\Hsp$ be the closed linear subspace of
$L^2(\Gc)$ generated by the atoms $\sett{\varphi_\lambda: \lambda \in
\Lambda}$.

Since $\Gc$
is now assumed to be unimodular, left and right translations are isometries
on $L^2(\Gc)$. Hence, the weight $w$ is also admissible
for $L^2(\Gc)$ (cf. Equation \ref{weight_w_admissible}) and
consequently the operators $\coef$ and $\rec$ 
from Section \ref{sec_setting_atdesc} map $L^2(\Gc)$ into
$\ell^2(\Lambda)$ and $\ell^2(\Lambda)$ into $L^2(\Gc)$, respectively
(cf. Equations \eqref{eq_coef} and \eqref{eq_rec}).
For clarity, when
considered with this domain and codomain we will denote these operators
by $\coef_\Hsp$ and $\rec_\Hsp$. Their adjoints will be denoted by
$\coef_\Hsp^*$ and $\rec^*_\Hsp$. Note that these operators coincide
with the maps $\coefp$ and $\recp$ on the intersection of their domains
(cf. Equations \eqref{eq_coefp} and \eqref{eq_recp}).
We also consider the operator $P_\Hsp := \rec_\Hsp \coef_\Hsp$,
which coincides with $P$ on $L^2(\Gc) \cap \Esp$.

Recall that a \emph{frame} for a Hilbert space $\Lsp$ is a collection of
vectors $\sett{e_k}_k$ such that
$\norm{v}_\Lsp \approx \norm{(\ip{v}{e_k})_k}_{\ell^2}$, for $v \in \Lsp$.
For a general reference on Hilbert-space frames see \cite{yo01, ch03}.
We now observe that the atoms of $\SEsp$ form a frame for $\Hsp$.
\begin{claim}
\label{claim_atoms_frame}
The set $\sett{\varphi_\lambda: \lambda \in \Lambda}$
is a frame for $\Hsp$.
\end{claim}
\begin{proof}
Since $f=P(f)=P_\Hsp(f)=\rec_\Hsp \coef_\Hsp (f)$
for finite linear combinations of the atoms
$\sett{\varphi_\lambda}_\lambda$, and $\coef_\Hsp$ and $\rec_\Hsp$
are bounded, it follows that $f=\rec_\Hsp \coef_\Hsp (f)$,
for all $f \in \Hsp$. This implies that
$f=Q \coef_\Hsp^*\rec_\Hsp^* (f)$, for all $f \in \Hsp$,
where $Q$ is the orthogonal projection onto $\Hsp$.
Hence, $\norm{f}_{L^2(\Gc)} \approx \norm{\rec^*_\Hsp(f)}_{\ell^2(\Lambda)}
= \norm{(\ip{f}{\varphi_\lambda})_\lambda}_{\ell^2(\Lambda)}$,
for all $f \in \Hsp$.
\end{proof}
Since $\sett{\varphi_\lambda: \lambda \in \Lambda}$ is a frame
for $\Hsp$, it has an associated canonical dual frame, that provides
an expansion with coefficients having minimal $\ell^2$-norm
(see for example \cite{yo01, ch03}).
This dual frame does not need to coincide with our distinguished
set of dual atoms 
$\sett{\psi_\lambda: \lambda \in \Lambda}$. We will now assume
that they do coincide. This assumption will be justified in a
large number of examples.
\begin{itemize}
 \item[(C2)] We assume that the set
$\sett{\psi_\lambda: \lambda \in \Lambda}$
is the canonical dual frame of
$\sett{\varphi_\lambda: \lambda \in \Lambda}$,
considered as a frame for $\Hsp$.
\end{itemize}
Under Assumption (C2), $\sett{\psi_\lambda}_\lambda \subseteq \Hsp$
and the operator $P_\Hsp$ is the orthogonal projector
$L^2(\Gc) \to \Hsp$. Also, $\coef_\Hsp$ and $\rec_\Hsp$ are related by
$\coef_\Hsp^\dag=\rec_\Hsp$ and
$\rec_\Hsp^\dag=\coef_\Hsp$. (Here $L^\dag$ denotes the
Moore-Penrose pseudo-inverse of an operator $L$).
\subsection{Convolution-dominated operators}
For the remainder of Section \ref{sec_more_general},
we assume that conditions (A1), (A2'), (B1), (C1) and (C2)
hold.

Using the fact that $\Lambda$ is a subgroup,
it is possible to dominate operators on $\Espd$ by convolutions.
We consider the class of operators dominated by left convolution,
\begin{align*}
CD(\Lambda, w)
:= \set{T \in \bC^{\Lambda \times \Lambda}}
{\abs{T_{\lambda, \lambda'}} \leq a_{\lambda \lambda'^{-1}}
\mbox{, for some $a \in \ell^1_w(\Lambda)$}},
\end{align*}
and we endow it with the norm,
\begin{align*}
\norm{T}_{CD(\Lambda, w)}
:= \inf \set{\norm{a}_{l^1_w}}
{\abs{T_{\lambda, \lambda'}} \leq a_{\lambda \lambda'^{-1}}
\mbox{, for all $\lambda, \lambda' \in \Lambda$}}.
\end{align*}
$CD(\Lambda, w)$ is a Banach *-algebra (see \cite{fegrle08}).
We also consider the Banach *-algebra of
operators dominated by right convolution,
\begin{align*}
CD_R(\Lambda, w)
:= \set{T \in \bC^{\Lambda \times \Lambda}}
{\abs{T_{\lambda, \lambda'}} \leq a_{\lambda'^{-1} \lambda}
\mbox{, for some $a \in \ell^1_w(\Lambda)$}},
\end{align*}
and we endow it with a norm in a similar manner.
We will use a slightly adapted version of the main result from
\cite{fegrle08}.
\begin{prop}
\label{prop_spec_cd}
The inclusion $CD_R(\Lambda, w) \hookrightarrow B(\ell^2(\Lambda))$
is spectral (i.e. it preserves the spectrum of each element).\footnote{Here,
$B(\ell^2(\Lambda))$ denotes the algebra of bounded operators
on $\ell^2(\Lambda)$.}
Moreover, if $L \in CD_R(\Lambda, w)$ is a self-adjoint
operator with closed range, then its pseudo-inverse
$L^\dag$ also belongs to $CD_R(\Lambda, w)$.
\end{prop}
\begin{proof}
Let $\Lambda^\textit{op}$ denote the set $\Lambda$ considered with
the opposite group operation, given by
$\lambda ._{op} \lambda' = \lambda' \lambda$. Since $x \mapsto x^{-1}$
is an algebraic and topological isomorphism between $\Lambda$
and $\Lambda^\textit{op}$ and the weight $w$ is symmetric,
it follows that $\Lambda^\textit{op}$ also satisfies the FGL-conditions with
respect to the restriction of the weight $w$. Hence, \cite[Corollary
6]{fegrle08} implies that $CD(\Lambda^\textit{op}, w)$ is a spectral
subalgebra of
$B(\ell^2(\Lambda^\textit{op}))$. Finally observe that
$CD_R(\Lambda, w)=CD(\Lambda^\textit{op}, w)$ and that
$B(\ell^2(\Lambda^\textit{op}))= B(\ell^2(\Lambda))$.

The second part of the theorem is a well-known consequence of the first one.
Since the inclusion $CD_R(\Lambda, w) \hookrightarrow B(\ell^2(\Lambda))$
is closed under inversion, it is also closed under holomorphic functional
calculus. For a self-adjoint operator with closed range $L \in CD_R(\Lambda,
w)$, its pseudo-inverse is given by $L^\dag=f(L)$, where $f(z)=z^{-1}$, for
$z \not=0$ and $f(0)=0$. $f$ is holomorphic on the spectrum of $L$ because,
since the range of $L$ is closed, 0 is an isolated point of its spectrum.
\end{proof}
\begin{rem}
\label{rem_cd_zd}
The result in \cite{fegrle08} seems to be the most appropriate one
for this context but in some cases it is also possible to apply
the results in \cite{su07-5,shisu09} to the same end. If the group
$\Lambda$ is $\Zdst$, then the desired result also follows from
\cite{ba90-1}, \cite{grle06} and \cite{sj95} with the advantage
of slightly relaxing the assumptions on the weight.
\end{rem}

We now observe that $CD_R(\Lambda, w)$ acts on $\Espd$.
\begin{prop}
\label{prop_conv_dom}
Let $T \in CD_R(\Lambda,w)$. Then the following holds.
\begin{itemize}
\item[(a)] $T$ maps $\Espd$ into $\Espd$ and
$\norm{T}_{\Espd \to \Espd} \lesssim\norm{T}_{CD_R(\Lambda,w)}$.
\item[(b)] $T: (\Espd,\ell^1_w) \to (\Espd,\ell^1_w)$ is continuous.
\end{itemize}
\end{prop}
\begin{proof}
Part (a) follows from Proposition \ref{Ed_is_right_invariant}
and the solidity of $\Espd$. For part (b), observe that the spaces
$L^1_w$ and $L^\infty_{1/w}$ satisfy the same assumptions that $\Esp$
(cf. Remark \ref{rem_lp_are_mod})
and consequently, by part (a), every operator in $CD_R(\Lambda,w)$ maps
$\ell^1_w$ into $\ell^1_w$ and $\ell^\infty_{1/w}$ into $\ell^\infty_{1/w}$.
Since the class $CD_R(\Lambda,w)$ is closed under taking adjoints
it follows that $T:\ell^\infty_{1/w} \to \ell^\infty_{1/w}$ is weak*
continuous, so part (b) follows.
\end{proof}
\subsection{Invertibility of multipliers}
We will now prove the invertibility of $\multm$ on $\SEsp$.
We assume that $m \in L^\infty(\Gc)$ is real-valued and satisfies,
\begin{align*}
0 < A \leq m \leq B < \infty,
\end{align*}
for some constants $A,B$, and we will establish a number of claims
that will lead to the desired conclusion.
\begin{claim}
\label{claim_Mn_inv}
The operator $\multm: \Hsp \to \Hsp$ is invertible.
\end{claim}
\begin{proof}
Observe that, since $P_\Hsp:L^2(\Gc) \to \Hsp$ is the orthogonal projector,
and $m$ is real-valued, 
the operator $\multm: \Hsp \to \Hsp$ is self-adjoint. Moreover,
for $f \in \Hsp$,
\begin{align*}
\norm{\multm(f)}_\Hsp \norm{f}_\Hsp
&\geq \ip{P(mf)}{f}
= \ip{mf}{f}
\\
&= \int_{\Gc} m(x) \abs{f(x)}^2 dx
\geq A \norm{f}_\Hsp^2.
\end{align*}
Hence, $\multm:\Hsp \to \Hsp$ is self-adjoint and bounded below
and therefore invertible.
\end{proof}
\begin{rem}
Claim \ref{claim_Mn_inv} may not be true without the assumption that
$m$ is nonnegative. Indeed, if $\Gc=\Rst$, $\Lambda=\Zst$, 
$\varphi_\lambda=\psi_\lambda=\chi_{[\lambda,\lambda+1]}$ and
$m=\chi_{(-\infty,1/2)}-\chi_{[1/2,+\infty)]}$, then $\multm(\varphi_0)=0$.
\end{rem}
Let $L \in \bC^{\Lambda\times\Lambda}$ be the matrix representing
the operator $\rec_\Hsp^* \multm \rec_\Hsp: \ell^2(\Lambda) \to
\ell^2(\Lambda)$.
Hence, $L$ is given by,
\begin{align*}
L_{\lambda, \lambda'} := \ip{m \varphi_{\lambda'}}{\varphi_\lambda}.
\end{align*}

\begin{claim}
\label{claimLdag}
The matrix $L$ belongs to $CD_R(\Lambda, w)$
and has a Moore-Penrose pseudo-inverse $L^\dag$
that also belongs to $CD_R(\Lambda, w)$.
In addition, $(\multm)^{-1}: \Hsp \to \Hsp$ can be decomposed as
$(\multm)^{-1} = \rec_\Hsp L^\dagger \rec_\Hsp^*$.
\end{claim}
\begin{proof}
To see that $L \in CD_R(\Lambda, w)$ let us estimate,
\begin{align*}
\abs{L_{\lambda, \lambda'}}
&\lesssim
\int_\Gc h(\lambda^{-1} x) h(\lambda'^{-1} x) dx
= a_{\lambda'^{-1} \lambda},
\end{align*}
where $a_\lambda := h * \iv{h}(\lambda)$.
Using Lemmas \ref{lemma_amfacts1} and \ref{lemma_amfacts2}
we see that $a \in \ell^1_w$.

The operator $\rec_\Hsp$ has range $\Hsp$ because
$\sett{\varphi_\lambda}_\lambda$ is a frame for $\Hsp$ (cf. Claim
\ref{claim_atoms_frame}). Since $\multm: \Hsp \to \Hsp$ is invertible
by Claim \ref{claim_Mn_inv}, the range of
$L= \rec_\Hsp^* \multm \rec_\Hsp$ equals $\rec_\Hsp^*(\Hsp)$. This
subspace is closed because $\rec_\Hsp^*$ is bounded below on $\Hsp$
(that is the frame condition). Hence, $L$ has closed range and consequently
has a pseudo-inverse $L^\dag$.
Since $\multm$ is self-adjoint, so is $L$. In addition, $L^\dag$
is given by,
\begin{align*}
L^\dag &= \coef_\Hsp (\multm)^{-1} \coef_\Hsp^*.
\end{align*}
Hence, $(\multm)^{-1} = \rec_\Hsp L^\dagger \rec_\Hsp^*$,
(where the operator $\rec_\Hsp^*$ is restricted to $\Hsp$).
Finally, by Proposition \ref{prop_spec_cd},
$L^\dag \in CD_R(\Lambda, w)$.
\end{proof}
Now we can prove the invertibility of $\multm$ on $\SEsp$.
\begin{prop}
\label{prop_inv_mult}
Let $m \in L^\infty(\Gc)$ be real-valued and satisfy
$0 < A \leq m \leq B < \infty$,
for some constants $A,B$.
Then, the multiplier $\multm: \SEsp \to \SEsp$ is invertible.
\end{prop}
\begin{proof}
Let $\multn: \SEsp \to \SEsp$ be the operator defined by
$\multn := \rec L^\dag \recp$ (cf. Equation \eqref{eq_recp}).
It follows from 
Proposition \ref{prop_conv_dom} and Claim \ref{claimLdag} that
$\multn$ is bounded. Moreover, by Claim \ref{claimLdag},
for $f \in \SEsp \cap \Hsp$,
\begin{align}
\label{eqMN}
\multm \multn (f) = \multn \multm (f) = f.
\end{align}

By Propositions \ref{prop_weak_cont} and \ref{prop_conv_dom},
the operators $\multm$ and $\multn$ are continuous in the
$(\Esp, W(L^\infty, L^1_w))$ topology.
Since by Proposition \ref{prop_weak_cont}, any $f \in \SEsp$
can be approximated by a net of elements of $\SEsp \cap \Hsp$
in the $(\Esp, W(L^\infty, L^1_w))$ topology
(by considering the partial sums of the expansion
in Equation \eqref{at_desc}) it follows that Equation \eqref{eqMN} holds
for arbitrary $f \in \SEsp$. Hence, $\multm: \SEsp \to \SEsp$ is invertible.
\end{proof}
\subsection{Characterization of the atomic space with multipliers}
Finally we can derive the extension of Theorem
\ref{th_coverings_do_cover} to more general partitions of unity.
\begin{theorem}
\label{th_coverings_do_cover_2}
Suppose that Assumptions (A1), (A2'), (B1), (C1) and (C2) are satisfied.
Let $\sett{\theta_\gamma: \gamma \in \Gamma}$ be
given by $\theta_\gamma = m \eta_\gamma$,
where $0 < A \leq m \leq B < \infty$.

Then the operator,
\begin{align*}
\widetilde{\coefv}:\SEsp &\to \EspdB(\Gamma)
\\
f &\mapsto (P(f \theta_\gamma))_\gamma
\end{align*}
is left-invertible. Consequently, the following norm equivalence
holds for $f \in \SEsp$,
\begin{align*}
 \norm{f}_\Esp \approx
\norm{
({\norm{P(f\theta_\gamma)}_\Bsp})_\gamma
}_{\Espd}.
\end{align*}
\end{theorem}
\begin{rem}
Any family $\sett{\theta_\gamma}_\gamma$ that
is enveloped by $g$ and whose sum is a real-valued function
that is bounded away from 0 and $\infty$, has the prescribed form
for an adequate choice of the partition of unity
$\sett{\eta_\gamma}_\gamma$
and the function $m$.
\end{rem}
\begin{rem}
As in Theorem \ref{th_coverings_do_cover},
the norm equivalence holds uniformly for any class of spaces
$\Esp$ having the same weight $w$ and the same constant $\consEw$ (cf.
Equation \eqref{weight_w_admissible}).
\end{rem}
\begin{proof}
First observe that $\widetilde{\coefv}(f)=\coefv(mf)$,
so $\widetilde{\coefv}$ is bounded on $\SEsp$ by Propositions
\ref{prop_P_into_am} and \ref{prop_vector_synt}.
By Proposition \ref{prop_inv_mult}, $\multm$ is invertible,
so by Theorem \ref{th_approx_mult}
we can choose a relatively compact neighborhood of the identity $U$
such that $\multmU$ is also invertible.
Since the operator $P_U$ (cf. Equation \eqref{deffPU})
can be factored as $P_U=\recv_U \coefv$, we have that,
$\multmU(f)=P \recv_U \coefv (mf) = P \recv_U \widetilde{\coefv}(f)$.
Since  $\multmU$ is invertible, $\widetilde{\coefv}$ is is left-invertible,
as claimed. This also implies the desired norm equivalence.
\end{proof}
\section{Applications}
\label{sec_applications}
\subsection{Coorbit spaces}
\label{sec_app_co}
We now briefly introduce coorbit theory (see \cite{fegr89}) and
show how Theorem \ref{th_coverings_do_cover} applies to this context.
Let $\pi$ be a (strongly continuous) unitary representation of $\Gc$ on a
Hilbert space
$\HH$. For a fixed $h \in \HH$, the abstract wavelet transform is defined as,
\begin{align*}
V_h f(x) := \ip{f}{\pi(x) h}, \quad
(f \in \HH, x \in \Gc).
\end{align*}
Let $w$ be an admissible weight on $\Gc$.
The main assumption in coorbit theory is the existence of a cyclic vector
$h \in \HH$ that is admissible in the sense that:
$V_h h \in W_R(L^\infty,L^1_w)$ and the reproducing formula,
\begin{align*}
V_h f = V_h f * V_h h,
\end{align*}
holds for all $f \in \HH$. (For a study about the validity
of the reproducing formula see \cite{fu05}).
Since $V_h h(x^{-1})=\overline{V_h h (x)}$, it follows
that $V_h h$ also belongs to $W(L^\infty,L^1_w)$.
As a consequence of the reproducing formula,
$V_h: \HH \to L^2(\Gc)$ is an isometry
and therefore has an inverse on its (closed) range.

Under these conditions the space $\HH_w^1$
is defined by,
\[
\HH_w^1 := \set{f \in \HH}{V_h f \in L^1_w},
\]
and endowed with the norm $\norm{f}_{\HH_w^1} := \norm{V_h f}_{L^1_w}$.
The anti-dual of $\HH_w^1$ (i.e. the space of continuous conjugate-linear
functionals) is denoted by $(\HH_w^1)^\urcorner$.
The inner product $\Hst \times \Hst \to \bC$ extends to a
sesquilinear form on $\HH_w^1 \times (\HH_w^1)^\urcorner \to \bC$.
Since $h$ is assumed to belong to $\HH_w^1$, the abstract wavelet
transform can be defined for $f \in (\HH_w^1)^\urcorner$.

Coorbit spaces are defined by selecting from the reservoir
$(\HH_w^1)^\urcorner$ those elements that satisfy a certain criteria.
Let $\Esp$ be a solid, translation invariant BF space such that $w$
is admissible for it. The coorbit space is defined by,
\[
\mathrm{Co} \Esp := \set{f \in (\HH_w^1)^\urcorner}{V_h f \in \Esp},
\]
and endowed with the norm $\norm{f}_{\mathrm{Co}\Esp} := \norm{V_h
f}_\Esp$.

Let $\SEsp = V_h (\mathrm{Co} \Esp)$. According to
\cite[Proposition 4.3]{fegr89},
$\SEsp$ is a closed subspace of $\Esp$ and moreover
$P(F):=F*V_hh$ defines a projector onto $\SEsp$. Hence,
if we let $H := \abs{V_h h}$, Assumptions (A1) and (A2) are verified.
When $\Esp$ is $L^2(\Gc)$, the operator $P$ is in fact the orthogonal
projector onto $\SEsp$.

In order to apply Theorem \ref{th_coverings_do_cover} to this setting,
let a partition of unity $\sett{\eta_\gamma}_\gamma$ and
a BF space $\Bsp$ satisfying (B1) be given. Let the operators
$M_\gamma:\mathrm{Co} \Esp \to \mathrm{Co} \Esp$ be defined as,
\begin{align*}
M_\gamma(f) := V_h^*(\eta_\gamma V_h(f)).
\end{align*}
Observe that, since $V_h:\HH \to L^2$ is an isometry,
$V_h^*$ is the projection onto the range of $V_h$ followed by the
inverse of $V_h$ on its range. Hence,
\begin{align*}
V_h M_\gamma(f) := M_{\eta_\gamma} V_h(f),
\end{align*}
where $M_{\eta_\gamma}: \SEsp \to \SEsp$ is the multiplier
form Section \ref{sec_approx_mult}. Now
Theorem \ref{th_coverings_do_cover} yields the following.
\begin{theorem}[Characterization of coorbit spaces]
\label{th_app_coorbit}
Let a partition of unity $\sett{\eta_\gamma}_\gamma$ and
a BF space $\Bsp$ satisfying (B1) be given. Then,
for $f \in \mathrm{Co} \Esp$, the following norm equivalence
holds,
\begin{align*}
\norm{f}_{\mathrm{Co} \Esp}
\approx
\norm{\sett{\norm{M_\gamma(f)}_{\mathrm{Co} \Bsp}}_\gamma}_{\Espd}.
\end{align*}
Moreover, the norm equivalence holds uniformly for any class of 
coorbit spaces $\mathrm{Co} \Esp$ having the same weight $w$ and
the same constant $\consEw$ (cf. Equation \eqref{weight_w_admissible}).

In addition, $f \in (\HH_w^1)^\urcorner$ belongs
to $\mathrm{Co} \Esp$ if and only if
$\sett{\norm{M_\gamma(f)}_{\mathrm{Co} \Bsp}}_\gamma \in \Espd(\Gamma)$.
\end{theorem}
\begin{rem}
One possible choice for $\Bsp$ is $L^2(\Gc)$ yielding
$\mathrm{Co} \Bsp = \HH$ (cf. \cite[Corollary 4.4]{fegr89}).
\end{rem}
\begin{proof}
The norm equivalence follows directly from Theorem
\ref{th_coverings_do_cover} and the fact that
$V_h: \mathrm{Co} \Esp \to \SEsp$ is an isometry. The ``in addition''
part follows from a standard approximation argument.
\end{proof}
\subsection{Time-Scale decompositions}
We now consider the affine group
$\Gc = \Rdst \times (0,+\infty)$,
where multiplication is given by $(x,s)\cdot(x',s') = (x+sx',ss')$.
Haar measure has density $dx \frac{ds}{s^{d+1}}$ and the modular
function is given by $\Delta(x,s)=s^{-d}$.
The affine group acts on $L^2(\Rdst)$ by translations and dilations,
\[
\pi(x,s) f(y)=s^{-d/2}f\left(\frac{y-x}{s}\right).
\]
The Wavelet transform associated with $\pi$ is,
\[
W_h f(x,s) = s^{-d/2}\int_{\Rdst} f(t)
\overline{h \left(\frac{t-x}{s}\right)}dt,
\]
for $f,h \in L^2(\Rdst)$,
whereas the inverse wavelet transform is given by,
\[
W_h^*F(x) =
\int_0^{+\infty} \int_\Rdst
F(y,s) 
\overline{h \left(\frac{x-y}{s}\right)} dx \frac{ds}{s^{\frac{3}{2}d+1}},
\]
for $F \in L^2(\Gc)$.\footnote{The integral converges in the weak-sense.
The possibility of evaluating it pointwise requires further
hypothesis.}

The wavelet multiplier with symbol $m \in L^\infty(\Gc)$ is given by,
\begin{align}
\label{eq_wavelet_mult}
\mathrm{WM}_m f(x) =
W_h^*(mW_hF),
\end{align}
for $f \in L^2(\Rdst)$.

The class of coorbit spaces for $\pi$ contains a large range of the
classical function spaces (see \cite{gr91}) including
the Besov and Triebel-Lizorkin spaces. We illustrate Theorem
\ref{th_app_coorbit} for homogeneous Besov spaces.
For $1 \leq p,q \leq +\infty$ and $\sigma \in \Rst$,
the homogeneous Besov space  $\dot{B}_{pq}^{\sigma}(\mathbb{R}^d)$ is
the set of all tempered distributions (modulo polynomials)
$f\in \mathcal{S}'/ \mathcal{P}(\Rdst)$ such that
\begin{align*}
\norm{f}_{\dot{B}_{pq}^{\sigma}}
:=
\left(
\sum_{j\in \Zst}2^{j\sigma q}
\norm{\mathcal{F}^{-1} (\varphi_j\mathcal{F}(f))}_{L^p}
\big\|^q
\right)^{1/q}
\end{align*}
is finite (with the usual modification for $q=\infty$),
where $\mathcal{F}$ is the Fourier transform and
$\sett{\varphi_j}_j$ is an adequate Schwartz class
partition of unity subordinated to dyadic crowns.
It is also usual to present these spaces in terms
of integrability of moduli of continuity rather than
frequency truncations. See \cite{tr83} for details.

One of Triebel's characterization of Besov spaces \cite{tr88} 
(see also \cite{gr91}) implies that\footnote{Triebel's result
implies that
$\norm{f}_{\dot{B}_{pq}^{\sigma}} \approx 
\norm{W_h f}_{L^{p,q}_{\sigma+d/2-d/q}}$ for an adequate
window function $h$. In \cite{fegr89} it is shown
that all admissible windows $h$ induce equivalent norms
in the coorbit space.}
$\dot{B}_{pq}^{\sigma}(\Rdst)=\mathrm{Co}(L^{p,q}_{\sigma+d/2-d/q}(\Gc))$,
where,
\[
\norm{F}_{L^{p,q}_\sigma}
=\left(\int_{0}^{+\infty}
\left(\int_\Rdst |F(x,s)|^p
dx\right)^{q/p}s^{-\sigma
q}\frac{ds}{s^{d+1}}\right)^{1/q}.
\]
As shown in
\cite[Section 4.2]{grpi09} the admissibility of the window $h$
is implied by the classical ``smooth molecule'' conditions
involving decay of derivatives and vanishing moments
(see \cite{frja85,frja90, frjawe91}). For example, any radial Schwartz
function $h$ with all moments vanishing is adequate.\footnote{To satisfy the
general
assumptions of Section \ref{sec_model} we can use the weight
$w(x,s) := \max\sett{s^{-\sigma}, \Delta(x,s)^{-1}s^{\sigma}}
= \max\sett{s^{-\sigma}, s^{d+\sigma}}$.}

In order to illustrate Theorem \ref{th_app_coorbit}, we consider a
covering of $\Rdst \times (0,+\infty)$
of the form,
\begin{align}
\label{eq_affine_covering}
U_{k,j} := 2^j((-1,1)^d+k) \times (2^{j-1},2^{j+1}),
\qquad (k \in \Zdst, j \in \Zst),
\end{align}
and let $\sett{\eta_{k,j}}_{k,j}$ be a (measurable) partition of
unity subordinated to it. The discrete norm of a sequence $\set{c_{k,j}}{k
\in \Zdst, j \in \Zst}$ associated with the space $L^{p,q}_\sigma$ and the
covering in Equation \eqref{eq_affine_covering} is
(see for example \cite{gr91, ul10}),
\begin{align*}
\norm{c}_{(L^{p,q}_\sigma)_d}
\approx
\left(
\sum_{j \in \Zst} 
2^{-jq(\sigma + d/q - d/p)}
\left(
\sum_{k \in \Zdst} \abs{c_{k,j}}^p 
\right)^{q/p}
\right)^{1/q}.
\end{align*}

We now obtain the following result.

\begin{theorem}
The quantity,
\begin{align*}
\left(
\sum_{j \in \Zst}2^{-j\sigma' q}
\left(\sum_{k \in \Zdst}
\norm{\mathrm{WM}_{\eta_{k,j}} f}_{L^2}^p \right)^{q/p}
\right)^{1/q},
\end{align*}
where $\sigma' := \sigma+d/2-d/p$,
is an equivalent norm on $\dot{B}_{pq}^{\sigma}$
(with the usual modifications when $p$ or $q$ are $\infty$).
\end{theorem}
\begin{rem}
Observe that Theorem \ref{th_app_coorbit}
also allows for non-compactly supported partitions of
unity, as long as its members are enveloped by a well-concentrated function.
Also observe that in the norm equivalence above we
can measure the norms of $\mathrm{WM}_{\eta_{k,j}} f$
in other Besov spaces besides $L^2$.
\end{rem}
\subsection{Time-Frequency decompositions}
\label{sec_app_gab}
For $f,h \in L^2(\Rdst)$, the
\emph{Short-Time Fourier Transform} (STFT)
(or \emph{windowed Fourier Transform}) is defined by,
\begin{align*}
\stft_h f(x,\varsigma) = \int_\Rdst f(y) e^{-2\pi i \varsigma y}
\overline{h(y-x)} dy.
\end{align*}
The translation and modulation operators are given by
$T_x f (y):=f(y-x)$ and $M_\varsigma f (y):=e^{2\pi i \varsigma y}f(y)$,
so that,
\begin{align}
\label{eq_def_stft}
\stft_h f(x,\varsigma) := \ip{f}{M_\varsigma T_x h}.
\end{align}
If $h$ is suitably normalized, $\stft_h:L^2(\Rdst) \to L^2(\Rtdst)$
is an isometry. The adjoint (inverse) STFT is given by,
\begin{align*}
\stft_h^*F(x) = \int_\Rtdst F(y,\varsigma) M_\varsigma T_y h(x) dyd\varsigma,
\end{align*}
so the localization operator with symbol $m \in L^\infty(\Rtdst)$
is given by,
\begin{align*}
\mathrm{H}_m f (x) &= \stft_h^* ( m \stft_h f)(x)
= \int_\Rtdst m(y,\varsigma) \stft_h f(y,\varsigma)
M_\varsigma T_y h(x) dyd\varsigma.
\end{align*}
If $h$ belongs to the Schwartz class, the definition in Equation
\eqref{eq_def_stft} extends to tempered distributions. Modulation
spaces are then defined by imposing integrability conditions of
the STFT. Let $w: \Rtdst \to (0,+\infty)$ be a submultiplicative, even
weight that satisfies the GRS condition: $\lim_{n \rightarrow \infty}
w(nx)^{1/n}=1$, for all $x \in \Rtdst$. Let $v: \Rtdst \to (0,+\infty)$
be a $w$-moderated weight; that is: $v(x+y) \lesssim w(x) v(y)$, for all
$x,y \in \Rtdst$. Assume further that $v$ is moderated by a polynomial
weight\footnote{This assumption is only made in order to define modulation
spaces as subsets of the class of tempered distributions. For a general weight,
the space $M^{p,q}_v$ has to be constructed as an abstract coorbit space.}.
For $1 \leq p,q \leq +\infty$, the
modulation space 
$M^{p,q}_v$ is defined as,
\begin{align*}
M^{p,q}_v := \set{f \in \mathcal{S}'(\Rdst)}{\stft_h f \in
L^{p,q}_v(\mathbb{R}^{2d})}
\end{align*}
where,
\begin{align*}
\norm{F}_{L^{p,q}_v}
=\left(
\int_\Rdst
\left(
\int_\Rdst |F(x,\varsigma)|^p v(x,\varsigma)^p
dx
\right)^{q/p}
d\varsigma
\right)^{1/q},
\end{align*}
with the usual modifications when $p$ or $q$ are $+\infty$.
$M^{p,q}_v$ is of course given the norm
$\norm{f}_{M^{p,q}_v} = \norm{\stft_h f}_{L^{p,q}_v}$.
For more details on the STFT and modulations spaces see \cite{gr01}.

After some normalizations and identifications,
modulation spaces can be regarded as coorbit spaces
of the Schr\"odinger representation of the Heisenberg group.
We chose however to consider them in the context of Section
\ref{sec_setting_atdesc}. For $h \in M^{1,1}_w$, $1 \leq p,q \leq \infty$,
and $w,v$ as above, we let $\Gc$ be $\Rdst \times \Rdst$,
$\Esp := L^{p,q}_v(\Gc)$ and $\SEsp := \stft_h (M^{p,q}_v)$.

For an adequate lattice\footnote{By
a lattice, we mean a full-rank lattice; i.e, a set of the form
$\Lambda=A \Zst^{2d}$, where $A$ is an invertible matrix.}
$\Lambda \subseteq \Rtdst$ the system
$\set{M_\varsigma T_x h}{(x,\varsigma) \in \Lambda}$ gives rise
to an atomic decomposition of $M^{p,q}_v$.
Moreover, on $M^2=M^{2,2}_1$
the dual atoms consist of the Hilbert-space dual frame of
$\set{M_\varsigma T_x h}{(x,\varsigma) \in \Lambda}$
and are of the form
$\set{M_\varsigma T_x \tilde{h}}{(x,\varsigma) \in \Lambda}$
for some function $\tilde{h} \in M^{1,1}_w$
(see \cite{fegr97,gr01}).
Hence, if we define $\varphi_{(x,\varsigma)} :=
\stft_h (M_\varsigma T_x h)$ and
$\psi_{(x,\varsigma)} := \stft_h (M_\varsigma T_x \tilde{h})$, the atoms
$\set{\varphi_\lambda}{\lambda \in \Lambda}$ and
dual atoms $\set{\psi_\lambda}{\lambda \in \Lambda}$ provide an atomic
decomposition for $\SEsp$.

Since $\Gc$ is abelian, left and right amalgam spaces coincide.
The envelopes for the atoms and dual atoms
are the functions $\abs{\stft_h h}$ and $\abs{\stft_h \tilde{h}}$.\footnote{
For simplicity Assumption (A2') requires the same envelope
for both the atoms and the dual atoms, but clearly if they
have different envelopes then their sum serves as a common envelope.}
These functions indeed envelope the atoms because of the 
straightforward relation:
$\abs{\stft_h (M_\varsigma T_x f)} = \abs{\stft_h f(\cdot - (x,\varsigma))}$
(see \cite[Equation 3.14]{gr01}).
The fact that $h$ and $\tilde{h}$ belong to $M^{1,1}_w$ means
that $V_h h$ and $V_h \tilde{h}$ belong to $L^1_w$, but it is well-know
that in this case they also belong to $W(L^\infty,L^1_w)$
(see \cite[Proposition 12.1.11]{gr01}). This fact can also be derived
from the norm equivalence in Proposition \ref{prop_P_into_am}.

Let us now consider a family of functions $\set{\theta_\gamma}{\gamma \in
\Gamma}$ that satisfy
\begin{align*}
0 < A \leq \sum_\gamma \theta_\gamma \leq B < \infty.
\end{align*}
Let us also assume that $\Gamma$ is a relatively separated
subset of $\Rtdst$ and that there exists a function $g \in L^1_w(\Rtdst)$
such that $\abs{\theta_\gamma(x)} \leq g(x-\gamma)$, for all $x \in \Rtdst$
and $\gamma \in \Gamma$. We will let the space $\Bsp$ that measures
the localized pieces be an unweighted Lebesgue space $L^{r,s}$.
We are then in the situation of Section \ref{sec_more_general}
(remember that, since $\Gc$ is abelian, $L^1_w = \wweak$ -
cf. Proposition \ref{prop_weak_st_embeddings}).

To illustrate Theorem \ref{th_coverings_do_cover_2}
more clearly we further assume that $\Gamma = \Gamma_1 \times \Gamma_2$
for two relatively separated sets $\Gamma_1, \Gamma_2 \subseteq \Rdst$.
Then we get the following.
\begin{theorem}
\label{th_app_gabor}
For all $1 \leq s,t \leq \infty$, the quantity,
\begin{align*}
\left(
\sum_{\gamma_2 \in \Gamma_2}
\left(
\sum_{\gamma_1 \in \Gamma_1}
\norm{\mathrm{H}_{\theta_{(\gamma_1,\gamma_2)}} f}_{M^{s,t}}^p 
v(\gamma_1,\gamma_2)^p \right)^{q/p}
\right)^{1/q},
\end{align*}
is an equivalent norm on $M^{p,q}_v$
(with the usual modifications when $p$ or $q$ are $\infty$).
\end{theorem}
This generalizes the main result in \cite{dogr09} in two directions.
The results in \cite{dogr09} apply only to partitions of unity produced
by lattice translations of a single function, whereas Theorem
\ref{th_app_gabor} allows for irregular partitions.
Secondly, in \cite{dogr09} the space measuring the localized pieces is
restricted to be $L^2$. In contrast, in Theorem \ref{th_app_gabor} it is
possible to measure the localized pieces using the whole range of unweighted
modulations spaces.

The proof in \cite{dogr09} resorts to techniques from rotation algebras and
spectral theory to construct an atomic decomposition that is simultaneously
adapted to all the localization operators
$\set{\mathrm{H}_{\theta_\gamma}}{\gamma \in \Gamma}$.
Part of our motivation came from
the observation that such an atomic decomposition could be obtained in
a more constructive manner by using the technique of \emph{frame surgery},
recently introduced in \cite{ro10}.
\subsection{Shearlet spaces}
Theorem \ref{th_coverings_do_cover} can also be applied to the
recently introduced shearlet coorbit spaces \cite{dakustte09}.
For $a \in \Rst^* := \Rst \setminus \sett{0}$ and $s \in \Rst$, the
\emph{parabolic scaling} $A_a$ and the \emph{shear} $S_s$ are defined by,
\begin{align*}
A_a := \left[
\begin{array}{cc}
a & 0\\
0& \textit{sgn}(a)\sqrt{\abs{a}}
\end{array}
\right],
\qquad
S_s :=
\left[
\begin{array}{cc}
1 & s\\
0 & 1
\end{array}
\right].
\end{align*}
The shearlet group is the set $\Gc := \Rst^* \times \Rst \times \Rst^2$,
together with the operation,
\begin{align*}
(a,s,t)\cdot(a',s',t') := (aa',s+\sqrt{\abs{a}}s',t+S_sA_at').
\end{align*}
The shearlet group acts on $L^2(\Rst^2)$ by,
\begin{align*}
\pi(a,s,t) f(x) := \abs{a}^{-3/4}f(A_a^{-1}S_s^{-1}(x-t)). 
\end{align*}
In \cite{dakustte09} it is proved that any Schwarz function $h$
with Fourier transform supported on a compact subset of
$\Rst^*\times\Rst$ is an admissible window. The corresponding
wavelet transform,
\begin{align*}
SH_h f (a,s,t) := \ip{f}{\pi(a,s,t)h},
\end{align*}
is called the \emph{continuous shearlet transform}. 
Using Theorem \ref{th_coverings_do_cover}, shearlet coorbit spaces
can be described in terms of multipliers of the continuous shearlet
transform. See \cite{dakustte09} for the relevant explicit calculations
on the shearlet group (e.g. description of relatively separated sets).
\subsection{Localized frames}
Let us briefly point out that Theorem \ref{th_coverings_do_cover_2}
also applies to coorbit spaces of localized frames
(see \cite{gr04-1, fogr05, bacahela06}). If $\Hsp$ is a Hilbert space and
$\mathcal{F}=\sett{f_k}_{k \in \Zdst}$ is a frame for it,
every element in $f \in \Hsp$ has an expansion $f = \sum_k \ip{f}{f_k} g_k$,
where $\sett{g_k}_{k \in \Zdst}$ is the so-called canonical dual frame.
$\mathcal{F}$ is said to be \emph{localized} if $\abs{\ip{f_k}{f_j}} \leq
a_{k-j}$, for some sequence $a \in l^1_w(\Zdst)$ and a weight like the one
considered in Section \ref{sec_more_general}. Frame multipliers are defined
by applying a mask to the frame expansion:
\[
M_m(f) :=  \sum_k m_k \ip{f}{f_k} g_k,
\]
where $m \in l^\infty(\Zdst)$. Coorbit spaces $H^p_v(\mathcal{F})$
are defined by imposing $\ell^p_v$ summability conditions to
the coefficients $\ip{f}{f_k}$ (see \cite{fogr05} for the details). Theorem
\ref{th_coverings_do_cover_2} can be applied using $\Gc=\Lambda=\Zdst$
and yields a characterization of the spaces $H^p_v(\mathcal{F})$
in terms of frame multipliers. Hence, for example, modulation spaces (cf.
Section \ref{sec_app_gab}) can also be characterized in terms of the
so-called Gabor multipliers \cite{feno03}.

\section{Acknowledgements}
The author thanks Monika D\"orfler, Hans Feichtinger, Karlheinz Gr\"ochenig,
Franz Luef and Ursula Molter for their insightful comments and suggestions.

\bibliographystyle{abbrv}

\end{document}